\documentclass[11pt]{article}
\usepackage{graphicx,color,enumerate}
\usepackage{soul}
\definecolor{refkey}{rgb}{1, 0.5, 0}  
\definecolor{labelkey}{rgb}{0.1,1,0.2}
\usepackage{hyperref}

\usepackage{amsmath,amsfonts,amsthm,amssymb}

\def\ds{\displaystyle}
\def\forall{\hbox{for all}~}
\def\L{{\bf L}}

\def\ve{\varepsilon}
\def\n{\noindent}
\def\A{{\cal A}}

\def\D{{\mathcal D}}

\def\R{{\mathbb R}}

\def\implies{\Longrightarrow}

\def\tv{\hbox{Tot.Var.}}

\def\v{\vskip 1em}

\def\C{{\bf C}}
\def\K{{\mathcal K}}
\def\F{{\mathcal F}}

\def\bega{\begin{array}}
\def\enda{\end{array}}
\def\begi{\begin{itemize}}
\def\endi{\end{itemize}}
\def\ov{\overline}
\def\wto{\rightharpoonup}
\def\Tilde{\widetilde}
\def\Hat{\widehat}

\def\bel{\begin{equation}\label}
\def\eeq{\end{equation}}
\def\sqr#1#2{\vbox{\hrule height .#2pt
\hbox{\vrule width .#2pt height #1pt \kern #1pt
\vrule width .#2pt}\hrule height .#2pt }}

\newtheorem{theorem}{Theorem}[section]

\newtheorem{corollary}[theorem]{Corollary}
\newtheorem{lemma}[theorem]{Lemma}
\newtheorem{definition}[theorem]{Definition}

\begin{document}

\title{\bf Vanishing Viscosity Solutions for Conservation Laws with Regulated Flux}

\author{Alberto Bressan$^*$, Graziano Guerra$^{**}$,
and Wen Shen$^*$
\\
\, \\
{\small (*) Department of Mathematics, Penn State University,
University Park, PA 16802, U.S.A.}\,\\ 
\small (**) Department of Mathematics and its Applications,
  University of Milano - Bicocca. \\ \, \\
{\small E-mails:  axb62@psu.edu, ~ graziano.guerra@unimib.it,~
wxs27@psu.edu}
\,\\}
\maketitle

\begin{abstract}
In this paper we introduce a concept of ``regulated function" $v(t,x)$ of two variables,
 which reduces to the classical definition when $v$ is independent of $t$.
We then consider a scalar conservation law of the form 
$u_t+F(v(t,x),u)_x=0$, where $F$ is smooth
and $v$ is a regulated function, possibly discontinuous w.r.t.~both $t$ and $x$.  
By adding a small viscosity, one obtains a well posed parabolic equation. 
As the viscous term goes to zero, the uniqueness of the vanishing viscosity limit is proved,
relying on comparison estimates for solutions to the corresponding
Hamilton--Jacobi equation.

As an application, we obtain the existence and uniqueness of solutions 
for a class of $2\times2$ 
triangular systems of conservation laws with hyperbolic degeneracy.
\end{abstract}

\medskip


\medskip



\section{Introduction}
\label{sec:1}
\setcounter{equation}{0}

We consider  the Cauchy problem for a scalar conservation law
of the form
\begin{equation}
\begin{cases} 
u_t+ F(v(t,x), u)_x~=~0,\\
u(0,x)~=~u_0 (x),
\end{cases}
\label{1}
\end{equation}
where the flux function $F$ is continuously differentiable but the 
function $v$ can be discontinuous w.r.t.~both variables $t,x$.
Our main concern is the convergence of the viscous approximations
\begin{equation}
 u_t+ F(v(t,x), u)_x~=~\varepsilon\, u_{xx},\label{3}
\end{equation}
to a unique weak solution to~\eqref{1}, as the viscosity parameter
$\varepsilon\to 0$.  

Starting with the works by N.~Risebro and collaborators~\cite{GR1, GR2, KR, KR1}, 
scalar conservation laws with discontinuous coefficients have now become 
the subject of an extensive literature~\cite{AG1, AG2, AKR1, AP, BV, CR,  Diehl, 
GNPT, KR2, Mit, SV}, 
also including some multi-dimensional cases~\cite{AKR2, CCP, CCPG}.

Results on the uniqueness and stability of vanishing viscosity solutions  have been obtained mainly
in the case where $v=v(x)$ is  piecewise smooth
with finitely many jumps.  Aim of this paper is to develop an alternative approach, based
on comparison estimates for solutions to the corresponding Hamilton--Jacobi equation.
This will yield the uniqueness of the vanishing
viscosity limit under the more general assumption that $v$ is  a ``regulated" function
of the two variables $t$ and $x$.
We recall that a function of a single variable 
$v:\mathbb{R}\mapsto\mathbb{R}$ is {\it regulated}
if it admits left and right limits at every point.
This is true if and only if, for every interval $[x_1, x_2]$ and every $\varepsilon>0$,
there exists a 
piecewise constant function $\chi$ such that 
\begin{equation}
\|\chi-v\|_{\L^\infty([x_1, x_2])}~\leq~\varepsilon\,.\label{4}
\end{equation}
We extend this concept to functions of two variables, as follows.
\v

\begin{definition}\label{def:1}
 We say that a bounded 
function $v=v(t,x)$ is {\bf regulated} if, for every intervals $[x_1,
x_2]$ and $[0,T]$, and any
$\varepsilon>0$, the following holds.

There exist finitely many disjoint subintervals $[a_i, b_i]\subseteq [0,T]$,
Lipschitz continuous curves 
\[
\gamma_{i,1}(t)~<~\gamma_{i,2}(t)~<~\cdots~<\gamma_{i, N(i)}(t),\qquad\quad t\in 
[a_i, b_i]\,,
\]
and constants $\alpha_{i,0}, \alpha_{i,1},\ldots,\alpha_{i, N(i)}$
such that 
\begin{itemize}
\item[(i)] For every $t\in [a_i, b_i]$, the step function
  \begin{equation}\label{5}
  \chi_i(t,x)\,\doteq\,
  \begin{cases}
  \alpha_{i,0}, &\hbox{if}~ x<\gamma_{i,1}(t),\\
  \alpha_{i,k},
 &\hbox{if}~ \gamma_{i,k}(t)<x<\gamma_{i,k+1}(t),\qquad k=1,2,\ldots, N(i)-1,\\
\alpha_{i, N(i)}, &\hbox{if}~ \gamma_{i,N(i)}(t)<x,
  \end{cases}
\end{equation}
satisfies
\begin{equation}\label{6}\|\chi_i(t,\cdot)- v(t,\cdot)\|_{\L^\infty([x_1,x_2])}~\leq~\varepsilon\,.
\end{equation}
\item[(ii)] For every $i,k$, 
the  time derivative $\dot\gamma_{i,k}(t)=
{d\over dt}\gamma_{i,k}(t)$
coincides a.e.~with a regulated function.
\item[(iii)] The intervals $[a_i, b_i]$ cover most of $[0,T]$, namely
\begin{equation}\label{8} T- \sum_{i}(b_i-a_i)~\leq~\varepsilon.\end{equation}
\end{itemize}
\end{definition}

We remark that, if $v=v(x)$ is independent of time,
then it satisfies  Definition~\ref{def:1} if and only if $v$ is a regulated function
in the usual sense.

\v
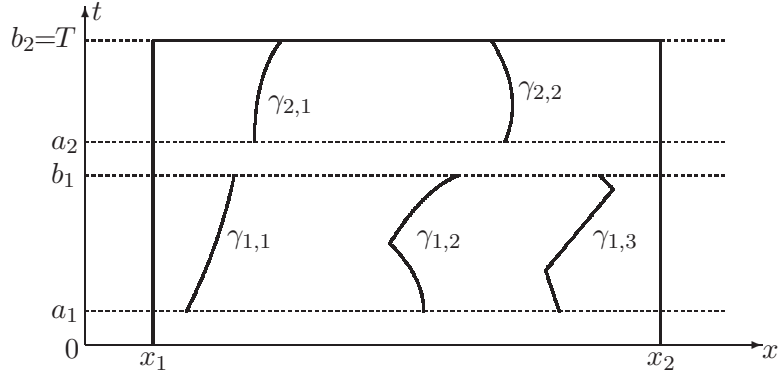
\begin{figure}[htbp]
\centering
\setlength{\unitlength}{0.9mm}
\begin{picture}(100,52)(0,-3) 
\put(0,0){\vector(1,0){100}}\put(100,-2){$x$}
\put(0,0){\vector(0,1){50}}\put(1,48){$t$}
\multiput(0,5)(1,0){95}{\line(1,0){0.5}} \put(-5,4){$a_1$}
\multiput(0,25)(1,0){95}{\line(1,0){0.5}}\put(-5,24){$b_1$}
\multiput(0,30)(1,0){95}{\line(1,0){0.5}}\put(-5,29){$a_2$}
\multiput(0,45)(1,0){95}{\line(1,0){0.5}}\put(-11,44){$b_2\!\!=\!\!T$}
\put(-3,-2){$0$} 
\thicklines
\put(10,0){\line(0,1){45}}\put(10,45){\line(1,0){75}}\put(85,0){\line(0,1){45}}
\put(8,-3){$x_1$}\put(83,-3){$x_2$}
\qbezier(15,5)(20,15)(22,25)
\put(21,15){$\gamma_{1,1}$}
\qbezier(25,30)(25,40)(29,45)
\put(27,35){$\gamma_{2,1}$}
\qbezier(50,5)(50,10)(45,15)\qbezier(45,15)(50,23)(55,25)
\put(49,15){$\gamma_{1,2}$}
\qbezier(70,5)(69,8)(68,11)\qbezier(68,11)(73,17)(78,23)\qbezier(78,23)(77,24)(76,25)
\put(75,15){$\gamma_{1,3}$}
\qbezier(62,30)(65,37)(60,45)
\put(64,37){$\gamma_{2,2}$}
\end{picture}
    \caption{\small According to Definition~1, a  {\bf regulated} function of two variables 
    $v=v(t,x)$ can
    be approximated by a piecewise constant function, with 
    jumps along finitely many Lipschitz curves $\gamma_{i,k}$.  The time derivatives $\dot \gamma_{i,k}$
   are regulated functions.}
\label{f:claw13}
\end{figure}

We shall study the convergence 
of the vanishing viscosity approximations~\eqref{3}, assuming
that $v$ is a regulated function.  Toward this goal, we also need
a standard assumption, which implies the uniform boundedness of viscous solutions. Namely:
\v  
\begi
\item[{\bf (A1)}] {\it  The values $F(\alpha, 0) ~=~h_0$ and
    $F(\alpha, 1) ~=~ h_1$
are independent of $\alpha$.}
\endi
For each $\ve>0$, let now
 $u^\ve=u^\ve(t,x)$ be a solution of~\eqref{3} taking values in $[0,1]$.
By extracting a suitable subsequence $\ve_n\to 0$ one achieves the weak convergence 
$u^{\ve_n}\wto u$ for some limit function $u$.

The main results in this paper show that
\begi
\item 
If $v=v(t,x)$ is a regulated function, then the weak limit
$u^{\varepsilon}\rightharpoonup u$
is unique.
Indeed, a comparison argument applied to the integrated functions 
\[
U^\ve(t,x)~\doteq~\int_{-\infty}^x u^\ve(t,y)\, dy
\]
shows that it converges uniformly on
$\left[0,T\right]\times\mathbb{R}$ as $\varepsilon\to 0$.
\item Under the additional assumption that for every rectangular
  domain of the form $\left[0,T\right]\times\left[x_{1},x_{2}\right]$
  one has
\begin{displaymath}
  \int_{0}^{T}\left(\tv\left\{v\left(t,\cdot\right);\left[x_{1},x_{2}\right]\right\}\right)\;
  dt ~<~+\infty,
\end{displaymath}
a compensated compactness argument implies that the unique weak limit $u$ is
a solution to the Cauchy problem~\eqref{1}.   In addition, if the partial derivative
$F_{\omega}\left(\alpha,\omega\right)$ does not vanish on any non-trivial interval 
$\left[\omega_{1},\omega_{2}\right]$, then the unique weak limit $u$ is
actually a strong limit.
\item 
If the function $v$ is obtained as the solution to a scalar conservation law:
\begin{equation}\label{cla}
v_t + g(v)_x~=~0, \qquad\qquad v(0,x) = v_0 (x),
\end{equation}
under quite general assumptions one can prove that $v$ is a regulated function.
The previous uniqueness results can thus be applied to a triangular system of the form
\begin{equation}\label{TS}
\left\{ \bega{rl} u_t+ F(v,u)_x&=~0,\\[1mm]
v_t + g(v)_x&=~0,\enda\right. \end{equation}
as the vanishing viscosity limit of the partially viscous system 
\[
\left\{ \bega{rl} u_t+ F(v,u)_x&=~\ve u_{xx},\\[1mm]
v_t + g(v)_x&=~0.\enda\right.
\]
\endi

Systems of conservation laws of the form~\eqref{TS}, which arise in a variety of applications~\cite{IT, Shen, TW},
were indeed the main motivation for the present study.

\medskip

The remainder of the paper is organized as follows. 
 In Section~\ref{sec:parab} we recall
some results on parabolic equations with singular coefficients and
prove some comparison results related to the corresponding
Hamilton--Jacobi equations. 
Section~\ref{sec:uniqueness} is the core of the paper, studying the class of fluxes for which 
the vanishing viscosity approximations have a unique 
weak limit.  We prove that this class 
includes all fluxes of the form $f(t,x,u)= F(v(t,x), u)$, 
where $F$ is a suitable smooth function and $v$ is regulated.  
 In Section~\ref{sec:compensated}, using a standard compensated compactness argument~\cite{Dafermos, KRT, Se}, 
 we prove that the unique limit is a weak solution to the corresponding conservation
 law. Under supplementary hypotheses we show
 the existence of a strong limit in $\L^1_{loc}$, for a sequence of  
 vanishing viscosity approximations.  Of course, the uniqueness of the weak limit implies that
 the strong limit is unique as well. Finally, Section~\ref{sec:reg} provides conditions which guarantee
 that the solution $v=v(t,x)$ of the  equation~\eqref{cla} is regulated.
Our analysis shows that this is the case if the flux function $g$   has at most one inflection point,
but may fail otherwise.
Some concluding remarks are given at the end in Section~\ref{sec:concluding}.

\section{Parabolic equations with discontinuous coefficients}
\label{sec:parab}
\setcounter{equation}{0}

In this section we consider a conservation law with discontinuous flux, 
in the presence of a fixed diffusion coefficient $\varepsilon>0$,
\begin{equation}
  \label{CPe}
  \begin{cases}
    u_t + f(t,x,u)_x~=~\varepsilon u_{xx}\,,\\
    u(0,x)~=~u_0 (x).
  \end{cases}
\end{equation}
In this case the equation is parabolic, and solutions can be represented as the fixed point of a 
strict contraction. The existence and uniqueness of solutions 
can be readily established, together with  their continuous dependence on the 
initial data and on the flux function.

If $f$ is smooth,
under mild hypotheses on the growth of the solution, this Cauchy problem 
is equivalent to the integral equation
\begin{equation}
  \label{eq:pebis}
  u(t,x)=\int_{\mathbb{R}} G^{\varepsilon}(t,x-y)\,u_0 (y)\, dy - \int_0^t\int_{\mathbb{R}} G^{\varepsilon}(t-s, x-y) f\left(s,y, u(s,y)\right)_{y}\, dy\, ds
\end{equation}
where, for $t>0$,
\begin{equation}\label{eq:Gs}
G(t,x)\;\doteq\;\frac{1}{\sqrt{4\pi t}}\, e^{-
    x^{2}/4 t},\qquad
  G^{\varepsilon}(t,x)\;\doteq\;\frac{1}{\sqrt{4\varepsilon\pi t}}\, e^{-
    x^{2}/4\varepsilon t}
\end{equation}
are the standard Gauss kernels. One has the identities
\begin{equation}
  \label{eq:nucleus_prop}
    \left\|G^{\varepsilon}(t,\cdot)\right\|_{\L^1}=1,\quad
    \left\|G_{x}^{\varepsilon}(t,\cdot)\right\|_{\L^1}= 2
    G^{\varepsilon}(t,0)=\frac{1}{\sqrt{\pi \varepsilon t}},
\end{equation}
for all $t>0$.  From 
 \eqref{eq:pebis}, an integration by parts yields
\begin{equation}
  \label{eq:fixedpoint}
  u(t,x)=\int_{\mathbb{R}} G^{\varepsilon}(t,x-y)\,u_0 (y)\, dy -
  \int_0^t\int_{\mathbb{R}} G_{x}^{\varepsilon}(t-s,\, x-y)
  f\bigl(s,y, u(s,y)\bigr)\, dy \,ds,
\end{equation}
which is meaningful even when $f$ is discontinuous.
Following~\cite{M, SY}, we say that $u=u(t,x)$ is a {\bf mild solution} of the Cauchy problem~\eqref{CPe} if it satisfies the integral identity~\eqref{eq:fixedpoint}.
A mild solution can thus be obtained as a fixed point
of  the transformation $u\mapsto {\mathcal P}^\ve u$, defined by
\begin{equation}
  \label{PT}
  \left(\mathcal{P}^{\varepsilon} u\right)(t,x)\;\doteq\;\int_{\mathbb{R}}
  G^{\varepsilon}
  (t,x-y)\,u_0 (y)\, dy -
  \int_0^t\int_{\mathbb{R}} G_x^{\varepsilon}
  (t-s,\, x-y)f\bigl(s,y, u(s,y)\bigr)\, dy \, ds.
\end{equation}
Multiplying by test function and integrating by parts, it is clear that a mild solution also solves~\eqref{CPe}
in distributional sense.
\v 
Let $T>0$ be given and consider the open domain
$\Omega\doteq \left]0,T\right[\times\mathbb{R}$.
For future use,
we collect here various hypotheses that will be imposed on the 
flux function $f:\Omega\times \R\mapsto\R$. 
\begin{description}
\item[{\bf (F1)}]  The function $f$ satisfies:
  \begi
  \item[(i)]  For each fixed $\omega\in\R$, the map
    $\left(t,x\right)\mapsto f(t,x,\omega)$ is in
    $\L^{\infty}(\Omega)$.
  \item[(ii)]
    The map $\omega\mapsto f(t,x,\omega)$ is twice continuously
    differentiable for any $\left(t,x\right)\in\Omega$ and
    there exists a constant $L\ge 0$ independent of $(t,x)$ such that
        \begin{equation}\label{fLip}
        \bigl|f(t,x,\omega_{1})-f(t,x,\omega_{2})\bigr|~\le~
        L\left|\omega_{1}-\omega_{2}\right|\qquad
        \text{ for all }~\omega_{1},\omega_{2}\,.
    \end{equation}
  \item[(iii)]
    There exists a constant $L_{1}\ge 0$ such that, 
    \begin{displaymath}
        \int_{\mathbb{R}}\bigl|f(t,x,0)\bigr|\, dx~\le~
        L_{1}\qquad \hbox{for all}~~t\ge 0.
    \end{displaymath}
  \endi

  \item[{\bf (F2)}]  For every
    $(t,x)\in\Omega$, the function $f$ satisfies
     $f(t,x,0)=0$ and $f(t,x,1)=h(t)$ 
   for some    $h\in\L^{\infty}\left(\left]0,T\right[,\mathbb{R}\right)$.
    \item[{\bf (F3)}]
    The function $f$ has the form
 \begin{equation}\label{Fax}
 f(t,x,\omega)~=~F\bigl(v(t,x), \omega\bigr),
 \end{equation}
  where $F\left(\alpha,\omega\right)$ is Lipschitz continuous w.r.t.~$\alpha$ and twice continuously 
  differentiable w.r.t.~$\omega$ satisfying
  \begin{equation}
    \label{Fass}
    F(\alpha,0)=0,\quad
    F(\alpha,1)=h_1,\quad \text{ for any }\alpha\in\mathbb{R}
\end{equation}
  and $v$ is a regulated function.
\end{description}


\v

The following theorem provides the existence and uniqueness of mild
solutions to~\eqref{CPe} under the assumption
{\bf (F1)} on the flux $f$. Moreover, it yields the continuous dependence
of solutions w.r.t.~the initial data and the flux function.

\begin{theorem}  \label{t:21}
  Consider the Banach space
  $Y_{T}\doteq \C^0([0,T],\,\L^1(\R))$
  endowed with the supremum norm 
  \[
  \|u\|_{T}\doteq
  \sup_{t\in [0,T]} \left\|u(t)\right\|_{\L^1(\R)}.
  \]
  Let the flux function $f$ satisfy {\bf (F1)} and take $u_0 \in\L^1(\R)$.
  \begi
  \item[(i)]
    The transformation $\mathcal{P}^{\varepsilon}$ defined in~\eqref{PT}
    is a Lipschitz continuous map
    from $Y_{T}$ into $Y_{T}$. It has a unique fixed point which
    is the unique solution to~\eqref{CPe} in $Y_T$. 
  \item[(ii)] Consider a sequence of initial data $(u_0 ^{\nu})_{\nu\geq 1}$ 
  converging to $u_0 $ in $\L^1(\R)$, and a sequence of fluxes $(f^\nu)_{\nu\geq 1}$,
  all satisfying {\bf (F1)} with the same constants $L, L_1$, and such that 
  $f^\nu(\cdot,\cdot, 0)\to f(\cdot,\cdot,0)$ in $\L^1(\Omega)$ and
  $f^\nu(\cdot,\cdot, \omega)\to f(\cdot,\cdot,\omega)$ in $\L^1_{loc}(\Omega)$, for 
  every $\omega\in\R$.
Then     the corresponding solutions
    $u^{\nu}$ to
    \begin{equation}
      \label{eq:parab-eq-app}
      \begin{cases}
        u_t + f^{\nu}\left(t,x,u\right)_x~=~\varepsilon u_{xx}\,,\\
        u(0)= u_0 ^{\nu}\,,
      \end{cases}
    \end{equation}
    converge in $Y_{T}$ to the solution $u$
    of \eqref{CPe}.
     \endi
\end{theorem}
\v

\begin{proof}
 {\bf 1.} 
  Using the inequality
  \begin{displaymath}
    \bigl|f(s,y, u(s,y))\bigr|
   ~ \le~ L \bigl|u(s,y)\bigr|+\bigl|f(s,y,0)\bigr|
  \end{displaymath}
  together with \eqref{eq:nucleus_prop},
  for any $u\in Y_{T}$ and $0\le t\le T$ by the assumptions {\bf (F1)} we obtain
    \begin{displaymath}
      \bigl\|(\mathcal{P}^{\varepsilon}
          u)(t,\cdot)\bigr\|_{\L^{1}(\R)}
     ~ \le~ \|
        u_0 \|_{\mathbf{L}^{1}(\R)}
      + \frac{2\sqrt{t}}{\sqrt{\pi\varepsilon}}\bigl(L\|u\|_{T}+L_{1}\bigr).
    \end{displaymath}
    Hence $\left\|\mathcal{P}^{\varepsilon}u\right\|_{T}<+\infty$.
    The dominated convergence theorem and the continuity of translations
  in $\mathbf{L}^{1}$ imply that the map $t\mapsto
  \left(\mathcal{P}^{\varepsilon}u\right)\left(t,\cdot\right)$ is continuous from
  $[0,T]$ into
  $\mathbf{L}^{1}(\R)$. Hence
  $\mathcal{P}^{\varepsilon}$ maps $Y_{T}$ into itself.
  
  Next, for any two functions $u_{1},u_{2}\in Y_{T}$, the Lipschitz
  continuity of $f$ implies
  \begin{displaymath}
      \bigl\|\mathcal{P}^{\varepsilon} u_{1} -
      \mathcal{P}^{\varepsilon}
    u_{2}\bigr\|_{T}~\le~ \frac{2L}{\sqrt{\pi\varepsilon}}\sqrt{T}
    \left\|u_{1}-u_{2}\right\|_{T}.
  \end{displaymath}
  This proves that $\mathcal{P}^{\varepsilon}$ is a well defined
  Lipschitz continuous map from $Y_{T}$ into itself.
  Choosing  
  \begin{equation}\label{TT} \Tilde T~=~ \frac{\pi\varepsilon}{16L^{2}}\,,\end{equation} 
  the above estimate shows  that 
  $\mathcal{P}^{\varepsilon}$ is a strict contraction restricted to $Y_{\Tilde
    T}$. Therefore $\mathcal{P}^{\varepsilon}$ has a unique fixed point on
  $Y_{\Tilde T}$. By induction, the same argument can be repeated on the intervals
  $[\Tilde T, 2\Tilde T]$,  $[2\Tilde T, 3\Tilde T] \ldots$ , until a unique solution is constructed on the 
  entire interval $[0,T]$.  
  This  concludes the proof of (i).
\v
\n {\bf 2.}   Toward a proof of (ii), 
  let $u$ be the unique mild solution of~\eqref{CPe}. 
We claim that
\begin{equation}
  \label{eq:smooth_approx2}
  \lim_{\nu\to\infty}\int_{\Omega}
  \bigl|f^{\nu}(t,x,u(t,x))-
  f(t,x,u(t,x))\bigr|\,dt\,dx~ =~0.
\end{equation}
Indeed, for any given $\epsilon>0$ we can approximate $u$ with a simple function
$u_{\epsilon}=\sum_{i=1}^{N}\omega_{i}\, \chi_{\Omega_{i}}$, with
$\Omega_{i}$, $i=1,\ldots,N$ bounded, so that 
$$\left\|u-u_{\epsilon}\right\|_{\L^1(\Omega)}~<~
\epsilon.$$ 
Thanks to the uniform Lipschitz continuity of both $f$ and $f^{\nu}$ w.r.t.~$\omega$, one has
\begin{displaymath}
  \begin{split}
    \int_{\Omega} &\bigl|f^{\nu}\left(t,x,u(t,x)\right)-
      f\left(t,x,u(t,x)\right)\bigr|\, dt\,dx \\
    &\le~
    \int_{\Omega}
    \bigl|f^{\nu}\left(t,x,u_{\epsilon}(t,x)\right)-
      f\left(t,x,u_{\epsilon}(t,x)\right)\bigr|\,dt\,dx+ 2L\epsilon\\
    &\le ~   \sum_{i=1}^{N}
    \int_{\Omega_{i}}\bigl|f^{\nu}\left(t,x,\omega_{i}\right)-
      f\left(t,x,\omega_{i}\right)\bigr|\,dt\,dx+
    \int_{\Omega}\bigl|f^{\nu}\left(t,x,0\right)-
      f\left(t,x,0\right)\bigr|\,dt\,dx+2L\epsilon.
  \end{split}
\end{displaymath}
By the assumptions on the convergence $f^\nu\to f$, 
since the sets $\Omega_{i}$ are bounded, we can take the limit as
$\nu\to\infty$ in the previous inequality and obtain
\begin{displaymath}
    \limsup_{\nu\to+\infty}\int_{\Omega} \bigl|f^{\nu}\left(t,x,u(t,x)\right)-
      f\left(t,x,u(t,x)\right)\bigr|\,dt\,dx ~\le~ 2L\epsilon\,.
  \end{displaymath}
  Since $\epsilon>0$ was arbitrary, 
this implies~\eqref{eq:smooth_approx2}.
\v
\n{\bf 3.}
It is enough to prove (ii) on $Y_{\Tilde T}$, where the Picard maps
$\mathcal{P}^{\varepsilon,\nu}$ is a strict contractions:
\begin{equation}\label{Pctr}
\|\mathcal{P}^{\varepsilon,\nu} u - \mathcal{P}^{\varepsilon,\nu} v\|_{\Tilde T}~\leq~
{1\over 2} \|u-v\|_{\Tilde T}\,.\end{equation}
  Indeed, the convergence can  then
be proved by induction on any interval $[k\tilde
  T,\, (k+1)\Tilde T]$ up to time $T$. 

Call $\mathcal{P}^{\varepsilon}$ and
$\mathcal{P}^{\varepsilon,\nu}$ the maps associated respectively to
Cauchy problems~\eqref{CPe} and ~\eqref{eq:parab-eq-app}, and let
$u$, $u^{\nu}$ be the corresponding fixed points. 
 Applying the contraction mapping theorem and the identity
$\mathcal{P}^{\varepsilon}u=u$, by~\eqref{Pctr} for any $\epsilon_o>0$ we have the estimate
\begin{align}
\nonumber
      \|&u-u^{\nu}\|_{\Tilde T}~\le~ 2
          \left\|u-\mathcal{P}^{\varepsilon,\nu}u\right\|_{\Tilde T}
         ~ =~2\left\|\mathcal{P}^{\varepsilon}u-
          \mathcal{P}^{\varepsilon,\nu}u\right\|_{\Tilde T}\\
          \nonumber
      &\le~ 2\left\|u_0  -  u_0 ^{\nu}
      \right\|_{\mathbf{L}^{1}}
      +2\sup_{t\in [0,\Tilde T]}
      \int_{0}^{t}\int_{\mathbb{R}}\frac{
      \left|f\left(s,y,u(s,y)\right)-
        f^{\nu}\left(s,y,u(s,y)\right)\right|}{\sqrt{\pi \varepsilon
          \left(t-s\right)}}
      \,dy\,ds\\
      \nonumber
      &\le ~2\left\|u_0  - u_0 ^{\nu}
      \right\|_{\L^1}
      +\frac{2}{\sqrt{\pi \varepsilon  \epsilon_o}}
      \sup_{t\in [\epsilon_{o},\Tilde T]}
      \int_{0}^{t-\epsilon_o}\int_{\mathbb{R}} \bigl|f\left(s,y,u(s,y)\right)-
        f^{\nu}\left(s,y,u(s,y)\right)\bigr|\,
      dy\,ds\\
      \nonumber
      &\quad+2\sup_{t\in [0,\Tilde T]}
      \int_{\max\left\{t-\epsilon_o,0\right\}}^{t}\int_{\mathbb{R}}\frac{1}{\sqrt{\pi \varepsilon
          \left(t-s\right)}}
      \Big[2L\left|u(s,y)\right|+\left|f(s,y,0)\right|+\left|f^{\nu}\left(s,y,0\right)
        \right|\Big]
      dy\,ds\\\nonumber
      &\le ~2\left\|u_0  - u_0 ^{\nu}
      \right\|_{\mathbf{L}^{1}}
      +\frac{2}{\sqrt{\pi \varepsilon
          \epsilon_o}}
      \int_{\Omega}
      \bigl|f\left(s,y,u(s,y)\right)-
        f^{\nu}\left(s,y,u(s,y)\right)\bigr|\,
      dy\,ds\\\nonumber
      &\quad+2\sup_{t\in [0,\Tilde T]}
      \int_{\max\left\{t-\epsilon_o,0\right\}}^{t}\frac{1}{\sqrt{\pi \varepsilon
          \left(t-s\right)}}
        \left[\int_{\mathbb{R}}2L\left|u(s,y)\right|\,
      dy+ 2L_{1}\right]ds\\\nonumber
      &\le~ 2\left\|u_0  - u_0 ^{\nu}
      \right\|_{\mathbf{L}^{1}}
      +\frac{2}{\sqrt{\pi \varepsilon
          \epsilon_o}}
      \int_{\Omega}
      \bigl|f(s,y,u(s,y))-
        f^{\nu}(s,y,u(s,y))\bigr|\,
      dy\,ds\\
        &\quad+2\left[2L\left\|u\right\|_{\Tilde T}
          +2L_{1}\right]2\sqrt{\frac{\epsilon_o}{\pi\varepsilon}}\,.
          \label{eq:long_comp}
  \end{align}
  With the help of~\eqref{eq:smooth_approx2} we obtain
  \begin{displaymath}
    \limsup_{\nu\to+\infty} \|u-u^{\nu}\|_{\Tilde T}
  ~  \le
    ~8\Big(L\left\|u\right\|_{\Tilde T}
          +L_{1}\Big)\sqrt{\frac{\epsilon_o}{\pi\varepsilon}}\,.
  \end{displaymath}
  Since  $\epsilon_o>0$ was arbitrary, this implies $\ds\lim_{\nu\to+\infty}u^{\nu}=u$ in $Y_{\Tilde T}$,
concluding the
  proof of (ii).
\end{proof}

 \v  
 The previous convergence result applies, in particular, to the case where
 the functions $f^\nu$ are obtained from $f$ by a mollification.
  More precisely, let $\rho\in \C^\infty_c(\R)$ be a standard mollification kernel, so that 
  \[\rho\geq 0,  \qquad Supp(\rho)\subset [-1,1], \quad \mbox{and}\quad \|\rho\|_{\L^1}=1.\]
  As usual, we then
define the rescaled kernels
\[\rho_{\delta}(\xi)\doteq \delta^{-1} \rho(\delta^{-1} \xi).\]
For a flux function satisfying   {\bf (F1)}, we
 consider the smooth approximations:
  \begin{equation}
    \label{eq:flux_approx}
    f_{\delta}(t,x,\omega)~\doteq~\int_{\Omega}
    \rho_{\delta}(t-s)\,\rho_{\delta}(x-y)\,
    f(s,y,\omega)\,dy\,ds\,.
  \end{equation}
  The functions $f_{\delta}(t,x,\omega)$ are $\C^\infty$ in the variables
  $(t,x)$ and satisfy {\bf (F1)}, with uniform constants $L, L_1$. Choosing 
  a decreasing sequence $\delta_\nu\to 0$ and defining $f^\nu = f_{\delta_\nu}$, 
   the assumptions in Theorem~\ref{t:21} (ii) are then satisfied.

 If the flux function $f=f(t,x,u)$ satisfies the additional assumptions
{\bf (F2)}, then the above functions $f^\nu = f_{\delta_\nu}$ obtained by a mollification satisfy
  \begin{equation}
      \label{eq:f_with_f1_0}
      f^{\nu}\left(t,x,0\right)=0,\quad
      f^{\nu}\left(t,x,1\right)=
      h^{\nu}(t)\,\doteq\,\int_0^T\rho_{\delta_{\nu}}    \left(t-s\right)h(s)\; ds,
       \quad \forall (t,x)\in\Omega.
    \end{equation}

 \v
    By well known regularity results in the theory of  parabolic equations~\cite{H, Lu, M}, 
    if the flux function $f$ is smooth, then
      the mild solutions constructed in Theorem~\ref{t:21} 
      are classical solutions.
 Relying on the fact that 
 \begi
\item  classical solutions to \eqref{CPe} satisfy various comparison
 properties, and  
 \item mild solutions can be approximated by classical ones,
 \endi
the following theorems and corollaries 
show that similar comparison properties
are valid for mild solutions as well.  
In a later section, these properties will play a key role in proving uniqueness
  of the vanishing viscosity limit.
\v

  \begin{theorem}
    \label{th:comp1}
    Let $u$ and $v$ be two mild solutions of the parabolic equation in \eqref{CPe},
    with initial data $u_0 ,v_0 \in \L^1(\R)$. Assume that  the flux function 
     $f$ satisfies {\bf (F1)}.     Then the following properties hold.
    \begi
    \item[(i)] The total mass is conserved in time:
      \begin{equation}
        \label{eq:mass_conservation}
        \int_{\mathbb{R}}u(t,x)\,dx~=~\int_{\mathbb{R}}
        u_0 (x)\,dx\qquad\text{ for all }~t\ge 0.
      \end{equation}
    \item[(ii)]
       A comparison holds:
      \begin{equation}
        \label{eq:firstcomparison}
        u_0 ~\le~ v_0 \qquad\implies\qquad
        u(t,\cdot)~\le ~v(t,\cdot)\quad \text{ for all }~t\ge 0.
      \end{equation}
    \item[(iii)] The $\L^1$ distance between the two solutions is non-increasing in time:
      \begin{equation}
        \label{eq:contraction}
        \int_{\mathbb{R}}\left|u(t,x)-v(t,x)\right|\,dx
       ~ \le~ \int_{\mathbb{R}}\left|u_0 (x)-v_0 (x)\right|\,dx\qquad 
        \text{ for all } t\ge 0.
      \end{equation}    \endi
  \end{theorem}
\v

\begin{proof}
To prove (i) it suffices
     to integrate \eqref{eq:fixedpoint}, observing that
   $$\int_{\mathbb{R}}G^{\varepsilon}\left(t-s,x-y\right)\,dx~=~1,\qquad\qquad
   \int_{\mathbb{R}}G^{\varepsilon}_{x}\left(t-s,x-y\right)\,dx~=~0.$$
    \v
    \n To prove (ii),
   we   choose convergent sequences of smooth fluxes $f^{\nu}\to f$ and of smooth
      initial data   $u_0 ^{\nu}\to u_0 $, $v_0 ^{\nu}\to v_0 $, with
      $u_0 ^{\nu}\le  v_0 ^{\nu}$ for every $\nu\geq 1$.
      Since these are smooth solutions, a standard comparison theorem yields
            \begin{equation}\label{compuv}
        u^{\nu}(t,x)~\le~ v^{\nu}(t,x)\qquad\text{ for all }
        t\ge 0,\ x\in\mathbb{R}. 
      \end{equation}
      The result is proven by taking the limit as $\nu\to \infty$ in~\eqref{compuv}, 
      using Theorem~\ref{t:21}.
 \v
\n  To prove (iii),  consider the initial data
  $$u_{o,*}~\doteq~\min\left\{u_0 , v_0 \right\},\qquad
  \qquad u_0 ^{*}~\doteq~\max\left\{u_0 ,  v_0 \right\},$$ 
  and let $u_*(t,x), u^*(t,x)$ be the corresponding solutions. 
  Since
    $u_{o,*}\le u_0 ,\, v_0 \le u_0 ^{*}$, by the comparison property (ii)
     the corresponding solutions  satisfy
    \[u_{*}(t,x)~\le ~u(t,x),\, v(t,x)~\le~ u^{*}(t,x)
    \qquad \mbox{for all }~t,x\in\Omega.\]
    By the conservation
    property \eqref{eq:mass_conservation}, this implies
\begin{eqnarray*}
 &&     \hspace{-2cm}  \int_{\mathbb{R}}\left|u(t,x)-v(t,x)\right|\,dx~\le~
        \int_{\mathbb{R}}\bigl[u^{*}(t,x)-u_{*}(t,x)\bigr]\,dx\\
     &  =&\int_{\mathbb{R}}\bigl[u_0 ^{*}(x)- u_{o,*}(x)\bigr]\,dx
        ~=~\int_{\mathbb{R}}\left| u_0 (x)- v_0 (x)\right|\,dx,
\end{eqnarray*}
      completing the proof.
\end{proof}
    
\medskip

In the following, together with \eqref{CPe} we consider a second 
Cauchy problem with different flux and initial data:
 \begin{equation}\label{CP2} \begin{cases}
    u_t + f^\sharp(t,x,u)_x~=~\varepsilon u_{xx}\,,\\
    u(0,x)~=~u_0 ^\sharp(x).
  \end{cases}
\end{equation}

  \begin{theorem}
    \label{th:comp2}
    Let $u$ and $u^{\sharp}$ be two solutions of \eqref{CPe} and \eqref{CP2},
    respectively. Assume that $u_0 ,u_0 ^\sharp\in \L^1(\R)$ and that both 
    fluxes $f$ and
    $f^{\sharp}$ satisfy {\bf (F1)}. 
    Let $U$ and $U^{\sharp}$ be
    the integrated functions:
    \begin{equation}
      \label{eq:def_integrated}
      U(t,x)=\int_{-\infty}^{x}u\left(t,\xi\right)\,d\xi,\qquad
      U^{\sharp}(t,x)=\int_{-\infty}^{x}u^{\sharp}\left(t,\xi\right)\,d\xi.
    \end{equation}
    Then the following comparison property holds.
    \begi    \item[]  
      Let $[a,b]$ be an interval containing the
      range of $u^{\sharp}(t,x)$ and assume that \linebreak
      $\eta\in\L^\infty\bigl([0,T]\bigr)$  and the constant $\bar
      \eta\ge 0$
      satisfy 
      \begin{equation}
        \label{eq:hyp_on_etat}
        \begin{cases}
          f^{\sharp}(t,x,\omega)\le f(t,x,\omega) + \eta(t)\quad 
          & \text{for all}~(t,x,\omega)\in\,]0,T[\times\R\times [a,b],\\
          U(0,x)\le U^{\sharp}(0,x) + \bar \eta\quad\ \qquad & \text{for all}~x\in\R.
        \end{cases}
        \end{equation}
      Then, for all $t\in [0,T]$ and $x\in\R$, one has
      \begin{equation}
        \label{eq:integrated_comparison}
        U(t,x)~\le~ U^{\sharp}(t,x)+\bar \eta+\int_{0}^{t}\eta(s)\,ds
        \,.
      \end{equation}
      \endi
  \end{theorem}

\begin{proof}
    Take a decreasing sequence $\delta_{\nu}\downarrow 0$  and consider the mollifications
    \begin{displaymath}
      \begin{split}
        \eta^{\nu}\left(t\right)&=\int_{\mathbb{R}}\rho_{\delta_{\nu}}\left(t-s\right)
        \eta(s)\,ds,\quad \\
        u_0 ^{\nu}\left(x\right)&=\int_{\mathbb{R}}\rho_{\delta_{\nu}}\left(x-y\right)
        u_0 (y)\,dy,\\
        u_0 ^{\sharp,\nu}\left(x\right)&=\int_{\mathbb{R}}\rho_{\delta_{\nu}}\left(x-y\right)
        u_0 ^{\sharp}(y)\,dy.
      \end{split}
    \end{displaymath}
    Construct 
    the corresponding mollifications of the fluxes
    $f^{\nu}$, $f^{\sharp,\nu}$,
    so that the first inequality in~\eqref{eq:hyp_on_etat} remains valid 
    for the smooth approximations: 
    \begin{equation}
      \label{eq:hyp_on_etat_app1}
          f^{\sharp,\nu}(t,x,\omega)\le f^{\nu}(t,x,\omega) +
          \eta^{\nu}(t) \qquad \text{ for all
          }(t,x,\omega)\in\,]0,+\infty[\times\mathbb{R}\times [a,b].      \end{equation}
      Fix any $\eta_{1}>\bar \eta$. Then,  for all $\nu$ sufficiently large,
      by the second inequality
      in~\eqref{eq:hyp_on_etat} it follows
    \begin{equation}
      \label{eq:hyp_on_etat_app2}
          U^\nu(0,x)~\le ~U^{\sharp,\nu}(0,x) + \eta_{1}\qquad \text{ for all }
          x\in\mathbb{R}.
      \end{equation}      
 Let $u^{\nu}$ be the  corresponding solution to~\eqref{eq:parab-eq-app}, so that
    \begin{displaymath}
      u^{\nu}(t,\xi)=\int_{\mathbb{R}} G^{\varepsilon}(t,\xi-y)\,u_0 ^{\nu}(y)\, dy -
      \int_0^t\int_{\mathbb{R}} G_{x}^{\varepsilon}(t-s,\, \xi-y)
      f^{\nu}(s,y, u^{\nu}(s,y))\,dy\,ds\,.
    \end{displaymath}
    Integrating the above equation over the interval $]-\infty,x[$ one obtains
    \begin{displaymath}
      U^{\nu}\left(t,x\right)=\int_{\mathbb{R}}
      G^{\varepsilon}(t,x-y)\,U^{\nu}(0,y)\, dy -
      \int_0^t\int_{\mathbb{R}} G^{\varepsilon}(t-s,\, x-y)
      f^{\nu}(s,y, u^{\nu}(s,y))\,dy\,ds\,.
    \end{displaymath}
    Since $u^{\nu}$ and its integral $U^{\nu}$ are smooth, the above
    integral identity implies that $U^{\nu}$ is a smooth solution to
    the Hamilton--Jacobi equation 
    \begin{displaymath}
      U^{\nu}_{t}+f^{\nu}\left(t,x,U_{x}^{\nu}\right)~=~\varepsilon U^{\nu}_{xx}.
    \end{displaymath}
 Similarly, $U^{\sharp,\nu}$ solves
    \begin{displaymath}
      U^{\sharp,\nu}_{t}+f^{\sharp,\nu}(t,x,U_{x}^{\sharp,\nu})~=~\varepsilon
      U^{\sharp,\nu}_{xx}.
    \end{displaymath}
   Introduce the function 
   \[E^{\nu}(t)=\eta_{1}+
    \int_{0}^{t}\eta^{\nu}(s)\,ds,\]
    depending only on time.  Combining
    the above equations,
    we  obtain
    \begin{displaymath}
      \bigl[U^{\nu}-U^{\sharp,\nu}-E^{\nu}\bigr]_{t}+
      f^{\nu}\left(t,x,U_{x}^{\nu}\right)-f^{\sharp,\nu}\bigl(t,x,U_{x}^{\sharp,\nu}\bigr)
      +\eta^{\nu}~=~\varepsilon
      \bigl[U^{\nu}- U^{\sharp,\nu}-E^{\nu}\bigr]_{xx}\,.
    \end{displaymath}
    Define 
    \[W^{\nu}\doteq U^{\nu}-U^{\sharp,\nu}-E^{\nu}\]
    and introduce the Hamiltonian
    function
    \begin{displaymath}
      \mathcal{H}^{\nu}\left(t,x,\omega\right)~\doteq~f^{\nu}\bigl(t,x,\omega+
        u^{\sharp,\nu}(t,x)\bigr) -
      f^{\sharp,\nu}\bigl(t,x,u^{\sharp,\nu}(t,x)\bigr)+
      \eta^{\nu}(t).
    \end{displaymath}
  Observe  that $W^{\nu}$ is a smooth solution to a viscous
    Hamilton-Jacobi equation:
    \begin{equation}
        \label{eq:hamiltonJ}
        W_{t}+\mathcal{H}^{\nu}\left(t,x,W_{x}\right)=\varepsilon
        W_{xx},
      \end{equation}
      with $$W^{\nu}(0,x)~=~U^{\nu}(0,x)-
        U^{\sharp,\nu}(0,x)-\eta_{1}~\le~ 0.$$
      Because of~\eqref{eq:hyp_on_etat_app1}, we have
      \begin{displaymath}
        \begin{split}
          \mathcal{H}^{\nu}(t,x,0)&~=~f^{\nu}\bigl(t,x,
            u^{\sharp,\nu}(t,x)\bigr) -
          f^{\sharp,\nu}\bigl(t,x,u^{\sharp,\nu}(t,x)\bigr)+
          \eta^{\nu}(t)\\
          &~\ge~ \inf_{\omega\in
          \left[a,b\right]}
           \left[f^{\nu}\bigl(t,x,
            \omega\bigr) -
          f^{\sharp,\nu}\bigl(t,x,\omega\bigr)+
          \eta^{\nu}(t)
        \right] ~\ge ~ 0\qquad\text{ for all }(t,x)\in\Omega.
      \end{split}
    \end{displaymath}
    Therefore the function $\ov W\equiv 0$ is a super-solution
    to~\eqref{eq:hamiltonJ}. 
    A standard comparison argument  now yields
    \begin{displaymath}
    W^{\nu}(t,x)~=~U^{\nu}(t,x)-U^{\sharp,\nu}(t,x)-E^{\nu}(t)~\le ~0,
    \qquad \text{
      for all }\bigl(t,x\bigr)\in\Omega.
  \end{displaymath}
  Letting $\nu\to\infty$ we obtain
\[
        U(t,x)~\le~ U^{\sharp}(t,x)+\eta_1+\int_{0}^{t}\eta(s)\,ds
        \,.
\]
  Since this is valid for every 
  $\eta_{1}> \bar\eta$,  the theorem is proved.
\end{proof}
\v

Let $f=f(t,x,\omega)$ be a
 flux function  satisfying {\bf (F1)}.  
The Lipschitz  property~\eqref{fLip} suggests that, for vanishing viscosity limits $u^\ve\to u$,
the characteristic speed should be~$\leq L$.  In particular, for every limit solution $u$,
one expects a bound of the form
\[
\int_{-\infty}^{x_0  -L(t-t_0)} |u(t,y)|\, dy~\leq~\int_{-\infty}^{x_0}
|u(t_0,y)|\, dy.
\]
Indeed, bounds of this form are well known in the case of a smooth flux~\cite{K}.
As a straightforward consequence of the comparison
Theorem~\ref{th:comp2},
we now prove
a similar estimate for viscous solutions.
\v
\begin{corollary}
\label{cor1}
Let $f=f(t,x,\omega)$ be a flux function satisfying the assumptions
{\bf (F1)} and {\bf (F2)}. 
For $\ve>0$, let $u^\ve$ be the solution to~\eqref{CPe} with initial data satisfying 
$u_0 \geq 0$,
$u_0 \in \L^1(\R)$. 
Then, for any $t_0,\delta_0 \geq 0$, $t>t_0$
and $x_0\in\R$, one has
the bound
\begin{equation}\label{eb1}
\int_{-\infty}^{x_0 -\delta_{0} -L(t-t_0)} u^\ve(t,y)\, dy~\leq~\int_{-\infty}^{x_0}
u(t_0,y)\, dy + E_\ve(t-t_0, \delta_0),\end{equation}
where
\begin{equation}\label{EE}E_\ve(\tau, \delta_0)~\doteq~\|
  u_0 \|_{\L^1}
  \cdot \int_{\delta_0/\sqrt{\tau\varepsilon}}^{+\infty} 
  G(1,x)\, dx,\quad \tau >0,
\end{equation}
where $G$ is standard Gauss kernel in~\eqref{eq:Gs}.
\end{corollary}
 \v
 
 \begin{proof} Using the same approximation argument
as in the proof of Theorem~\ref{th:comp2} we can assume that
the flux and
the initial datum are smooth.
Consider the integrated function
 \[U^\ve(t,x) =\int_{-\infty}^x u^\ve(t,y)\, dy.\]
 Then $U^\ve$ is a sub-solution of 
 $$U_t - L\, U_x~=~\ve U_{xx}\,,$$
 so that $V^\ve(\tau,y)\doteq U^\ve(t_{0}+\tau, x_{0}+y-L\tau) $ is a subsolution to 
 $$V_{\tau}~=~\ve V_{yy}\,.$$
 Therefore, using the fact
 that $V^{\varepsilon}(t_0,\cdot)$ is monotone increasing, we have
\begin{eqnarray*} 
V^{\varepsilon}(\tau,y)  &\leq& \left(\int_{-\infty}^0+\int_0^{+\infty} \right)
G^{\varepsilon}(\tau, y-\xi) V^{\varepsilon}(0,\xi)\, d\xi\\
&\le&  V(0,0) + \Big(\sup_{\xi\ge 0} V^{\varepsilon}(0,\xi)\Big)\cdot
\int_0^{+\infty}  G^{\varepsilon}(\tau, \, \xi-y)\, d\xi \\
                         &\leq&U^{\varepsilon}(t_0,x_{0}) + \| u_0 \|_{\L^1}
                                \int_0^{+\infty} G^{\varepsilon}(\tau,\xi-y)\, dy\,.
\end{eqnarray*}
In terms of the function $U^\ve$, with $\tau=t-t_{0}$ and
$y=-\delta_{0}$ this yields
\[
U(t,~x_0-\delta_0-L(t-t_0)) ~\leq~U(t_0, x_0) +\| u_0 \|_{\L^1}\cdot \int_0^{+\infty}
G^{\varepsilon}(t-t_0, \delta_0+y)\, dy.
\]
Since 
\[
\int_0^{+\infty}
G^{\varepsilon}(t-t_0, \delta_0+y)\, dy = \int_{\delta_0/\sqrt{(t-t_0)\varepsilon}}^{+\infty} 
  G(1,x)\, dx,
\]
this proves~\eqref{eb1}.
\end{proof}
  
We observe that, for each fixed $\delta_0>0$,
the error term $E_\ve$ in~\eqref{EE} goes to zero as $\ve\to 0$, uniformly as  
$\tau$ ranges over any bounded interval $]0, T]$ and $\varepsilon$
ranges in $\left]0,1\right]$. 

The following Corollary shows that the set
$\left\{u^{\varepsilon}\left(t,\cdot\right)\right\}$ is
tight (as defined, for example,  in Chapter 5 of~\cite{Royden}).
\begin{corollary}
  \label{cor1bis}
  Let $f=f(t,x,\omega)$ be a flux function satisfying the assumptions
  {\bf (F1)} and {\bf (F2)} and $u_0\in \L^1(\R)$ with $u_0\geq 0$. 
  For any $\ve>0$, let $u^\ve$ be the solution
  to~\eqref{CPe}.  Then the set of functions
  $\bigl\{u^{\varepsilon}\left(t,\cdot\right):
    \varepsilon\in\left]0,1\right],\;t\in\left[0,T\right]\bigr\}$ is tight. 
    More precisely,
  for any $\delta>0$ there exists $M>0$ which depends only on
  $\delta$, $u_0$
  and $L$ such that 
  \begin{equation}
    \label{eq:tight}
    \int_{\mathbb{R}\setminus
      \left[-M,M\right]}u^{\varepsilon}(t,x)\;dx<\delta,\quad
    \text{ for all }t\in\left[0,T\right],\;\varepsilon\in\left]0,1\right].
  \end{equation}
\end{corollary}

\v

\begin{proof}
  Fix $\delta>0$ and chose $x_{0}<0$, $\delta_{0}>0$ such that
  \begin{displaymath}
    \int_{-\infty}^{x_{0}}u_0 (x)\; dx+
    \int_{-x_{0}}^{+\infty}u_0 (x)\; dx+\left\|
      u_0 \right\|_{\L^{1}}
    \int_{\delta_{0}/\sqrt{T}}^{+\infty}G\left(1,x\right)\;dx<\frac{\delta}{2},
  \end{displaymath}
  then define $M=-x_{0}+\delta_{0}+LT$ and apply Corollary~\ref{cor1}.
\end{proof}

\v

Next, we consider two flux functions, say $f$ and $\hat f$, 
both satisfying the assumptions {\bf (F1)} and {\bf (F2)},   which coincide
on the half line $\{x<0\}$.   Let $u^\ve$ be the solution to~\eqref{CPe} 
and let $\hat u^\ve$ be the solution to 
 \begin{equation}\label{CPh}
    \begin{cases}
      u_{t}+\hat f(t,x,u)_{x}~=~\varepsilon
      u_{xx},\\
      u(0,x)=u_0 (x).
    \end{cases}
  \end{equation}
  Notice that here we are taking the same initial data $u_0 \in \L^1(\R)$.
We seek an estimate on the difference $u^\ve-\hat u^\ve$, on a region 
of the form $\{ x< -Lt\}$.

\begin{corollary}
  \label{cor2}
  In the above setting, assume that the two fluxes $f,\hat f$ 
 satisfy
  {\bf (F1), (F2)}, and coincide for $x<0$.
 Then the difference between the corresponding solutions $u^\ve,\hat
 u^\ve$ 
 satisfies
 \begin{equation}\label{uhu}
 \left|\int_{-\infty}^{-Lt-\xi} \bigl(u^\ve(t,y) - \hat u^\ve(t,y)\bigr)\, dy \right|~\leq~
 4\|u_0 \|_{\L^1}\cdot \int_{\xi/\sqrt{t\varepsilon}}^{+\infty} G(1,y)\, dy
 \end{equation}
 for all $\xi>0$.
 \end{corollary}
 
\v
\begin{proof} Using the same approximation argument
as in the proof of Theorem~\ref{th:comp2} we can assume that
both the fluxes and
the initial datum are smooth.
  Subtracting \eqref{CPh} from \eqref{CPe}
  one finds that the difference
  $w^{\varepsilon}=u^{\varepsilon}-\hat u^{\varepsilon}$ satisfies
 \[
    \begin{cases}
      w_{t}^{\varepsilon} + g(t,x,w^{\varepsilon})_{x}~=~\varepsilon
      w^{\varepsilon}_{xx}\,,\\
      w^{\varepsilon}(0,x)~=~0,
    \end{cases}
    \]
where the flux function is
\[
    g(t,x,\omega)~=~f\bigl(t,x,\omega+\hat u^{\varepsilon}(t,x)\bigr)
    -\hat f\bigl(t,x,\hat u^{\varepsilon}(t,x)\bigr).
\]
 The integrated function
\[
W^\ve(t,x)~=~\int_{-\infty}^x w^\ve(t,y)\, dy
\]
 thus satisfies
 \begin{equation}\label{We}
    \begin{cases}
      W_{t}^{\varepsilon} + g(t,x,W_x^{\varepsilon})~=~\varepsilon
      W^{\varepsilon}_{xx}\,,\\
      W^{\varepsilon}(0,x)~=~0.
    \end{cases}
  \end{equation}
  Consider the auxiliary function $Z=Z(t,x)$,
  defined as the solution to the Cauchy problem
 \begin{equation}\label{Ze} 
   \begin{cases}
     Z_{t}^{\varepsilon} - L\,Z^\ve_x=\varepsilon
     Z^{\varepsilon}_{xx},\\
     Z^{\varepsilon}\left(0,x\right)=4\left\|u_0 \right\|_{\L1}
     \chi_{\left[0,+\infty\right[}(x),
   \end{cases}\text{ i.e. }
   Z^{\varepsilon}\left(t,x\right)=4\left\|u_0 \right\|_{\L1}
   \int_{0}^{+\infty}G^{\varepsilon}\left(t,y-x-Lt\right)\;dy.
\end{equation}
 Observing that 
 \begin{displaymath}
   \begin{cases}
     |W^{\varepsilon}(t,x)|~
     \leq~2\|u_0 \|_{\L^1(\R)}&\text{ for all }t>0,\,x\in\mathbb{R}\,,\\[8pt]
     \bigl|g(t,x,\omega)\bigr|~\leq~L |\omega|&\text{ for all
     }t>0,\,x<0\,,\\[8pt]    
     Z^{\varepsilon}\left(t,x\right)\ge 2\left\|u_0 \right\|_{\L^1}
     &\text{ for all
     }t>0,\,x\ge 0\,,\\    
   \end{cases}
\end{displaymath}
we conclude that $Z^\ve$ satisfies $W^{\varepsilon}\left(0,x\right)
\le Z^{\varepsilon}\left(0,x\right)$ for all $x\in\mathbb{R}$ and
provides  a super-solution to~\eqref{We} in
the region $x<0$,
while it satisfies $W^{\varepsilon}\left(t,x\right)\le
Z^{\varepsilon}\left(t,x\right)$ in the region $x\ge 0$. Hence 
\[
W^\ve(t,x)~\leq~Z^\ve(t,x)\qquad\forall t>0, \;x\in\mathbb{R}.
\]
Exchanging the role of $u^{\varepsilon}$ and $\hat u^{\varepsilon}$ we
obtain $\left|W^{\varepsilon}(t,x)\right| \leq Z^\ve(t,x)$ for all
$t>0$, $x\in\mathbb{R}$ which coincides with \eqref{uhu} with
the substitution $x\to -Lt-\xi$.
\end{proof}

\section{The unique weak vanishing viscosity limit}  
\label{sec:uniqueness}
\setcounter{equation}{0}

Let $f=f(t,x,\omega)$ be a flux function satisfying {\bf (F1)}, {\bf (F2)}, and consider 
the domain
\begin{equation}\label{Ddef}
\D\,\doteq\,\left\{ u\in \L^1(\R)\,;~~u(x)\in [0,1]\quad\hbox{for all }~x\right\}.\end{equation}
Let an initial data $u_0 \in \D$ and a time interval $[a,b]$ be given.
For any $\ve>0$,  by {\bf (F2)} and the analysis in the previous section, the solution
$u^{\varepsilon}(t,x)$ to the Cauchy problem
  \begin{equation}\label{CP4}
 \begin{cases}      u_{t}+f\bigl(t,x,u\bigr)_{x}=~\ve
      u_{xx}\,,\\
      u(a,x)~=~u_0 (x)\,,
    \end{cases}
\end{equation}
  satisfies $u(t,\cdot)\in \D$ for all $t\in [a,b]$.

  We now consider a family of solutions $u^\ve$
  to the same Cauchy problem~\eqref{CP4},
for different values of the diffusion parameter $\ve>0$.
Since all these solutions are uniformly bounded,
we can extract  a decreasing sequence $\ve_n\to 0$ 
such that the corresponding solutions $u^{\ve_n}$
converge weakly to some function $u$.    
The main goal of this section is to find conditions on the flux function $f$ that yield the
uniqueness of the weak limit $u^{\ve_n}\wto u$, independently of the  particular sequence 
$\ve_n\to 0$.  

\begin{lemma}
    \label{lm:compact}
  Consider a flux $f=f(t,x,u)$ defined for $t\in[0,T]$, satisfying
  {\bf (F1)} and  {\bf (F2)} and let $u^{\varepsilon}$ be 
  solutions to \eqref{CP4} with a fixed initial datum $u_{o<}\in \mathcal{D}$ and
  $\varepsilon >0$. 
  Then, for any $t>0$:
  \begin{enumerate}[(i)]
  \item
    the set  $\left\{u^{\varepsilon}\left(t,\cdot\right)\right\}_{\varepsilon>0}$ is 
    relatively compact in the weak topology of
    $\L^{1}\left(\mathbb{R},\mathbb{R}\right)$;
  \item
    given a subsequence $u^{\varepsilon_{n}}$, one has
    \begin{equation}
      \label{eq:correspondence}
      u^{\varepsilon_{n}}\left(t,\cdot\right)\rightharpoonup u(t,\cdot) \quad
      \text{if and only if}\quad
      U^{\varepsilon_{n}}\left(t,\cdot\right)\to U\left(t,\cdot\right)
      \text{ uniformly in }\mathbb{R}
    \end{equation}
    with
    \begin{equation}
      \label{Un}
U^\ve(t,x)\,\doteq\,\int_{-\infty}^x u^\ve(t,y)\, dy,\qquad
 U(t,x)\,\doteq\,\int_{-\infty}^x u(t,y)\, dy;
\end{equation}
\item
  if $U^{\varepsilon}$ converges uniformly to $U$ in
  $\left[a,b\right]\times \mathbb{R}$ then the map $t\mapsto
  u(t,\cdot)$ is continuous from $\left[a,b\right]$ into
  $\L^{1}\left(\mathbb{R},\mathbb{R}\right)$ endowed with its weak
  topology. 
  \end{enumerate}
\end{lemma}
\v

\begin{proof}
  The set
  $\left\{u^{\varepsilon}\left(t,\cdot\right)\right\}_{\varepsilon>0}$
  is bounded in $\L^{1}$ by $\left\|u_0 \right\|_{\L^{1}}$, it is
  uniformly integrable \cite[Chapter 5]{Royden} because it is bounded in $\L^{\infty}$ and 
  it is tight because of Corollary~\ref{cor1bis}.  
  Dunford--Pettis Theorem \cite[Theorem 247C]{Frebook}
  implies that it is weakly relatively compact in
  $\L^{1}$.
  
  Suppose $u^{\varepsilon_{n}}\left(t,\cdot\right)\rightharpoonup
  u(t,\cdot)$. Weak convergence of
  $u^{\varepsilon_{n}}\left(t,\cdot\right)$ implies pointwise
  convergence of
  $U^{\varepsilon_{n}}\left(t,\cdot\right)$ to $U(t,\cdot)$. Arzel\`a--Ascoli theorem
  implies the uniform convergence on compact sets.
  Fix $\delta>0$ ad using Corollary~\ref{cor1bis} choose $M>0$ such
  that
  \begin{displaymath}
    \int_{-\infty}^{-M}u^{\varepsilon}\left(t,x\right)\;dx
    +\int_{M}^{+\infty}u^{\varepsilon}\left(t,x\right)\;dx<\delta.
  \end{displaymath}
  This implies the inequalities
  \begin{displaymath}
    \begin{cases}
      \left|U^{\varepsilon}\left(t,x\right)-U\left(t,x\right)\right|\le
      2\delta&\text{ for }x\le -M,\\
      \left|U^{\varepsilon}\left(t,x\right)-U\left(t,x\right)\right|\le
      \left|U^{\varepsilon}\left(t,M\right)-U\left(t,M\right)\right|+
      2\delta&\text{ for }x\ge M,\\
    \end{cases}
  \end{displaymath}
  so that
  \begin{displaymath}
    \left\|U^{\varepsilon}\left(t,\cdot\right)-U\left(t,\cdot\right)
    \right\|_{\C^{0}\left(\mathbb{R}\right)}
    \le 2\delta + 
    \left\|U^{\varepsilon}\left(t,\cdot\right)-U\left(t,\cdot\right)
    \right\|_{\C^{0}\left(\left[-M,M\right]\right)}.
  \end{displaymath}
  This gives
  \begin{displaymath}
    \limsup_{n\to+\infty}
    \left\|U^{\varepsilon_{n}}\left(t,\cdot\right)-U\left(t,\cdot\right)
    \right\|_{\C^{0}\left(\mathbb{R}\right)}
    \le 2\delta,
  \end{displaymath}
  which proves the uniform convergence on all the real line 
  since  $\delta>0$ is arbitrary.

  Suppose now the uniform convergence of
  $U^{\varepsilon_{n}}\left(t,\cdot\right)$ to some function
  $U\left(t,\cdot\right)$. The sequence $u^{\varepsilon_{n}}(t,\cdot)$
  is weakly compact and if a subsequence converges weakly
  to some function $u(t,\cdot)$, it must coincide with
  $U_{x}\left(t,\cdot\right)$ because of the previous part.
  Hence all the sequence
  $u^{\varepsilon_{n}}\left(t,\cdot\right)$
  converges weakly to $u(t,\cdot)=U_{x}\left(t,\cdot\right)$. 

  Point \textit{(ii)} implies that the limit $U$ is given by~\eqref{Un}
  where $u(t,\cdot)$ is the weak limit of
  $u^{\varepsilon}\left(t,\cdot\right)$. Since
  $u^{\varepsilon}\in\C^{0}\left(\left[0,T\right],\L^{1}\left(\mathbb{R}\right)\right)$,
  $U^{\varepsilon}\left(t,x\right)$ is continuous w.r.t.~both its variables. 
  Uniform convergence implies the continuity of the limit $U\left(t,x\right)$ on both
  its variables. Therefore the map
  $g_{\varphi}(t)=\int_{\mathbb{R}}\varphi(y)
  u(t,y)\;dy$ is continuous if $\varphi =
  \chi_{\left]-\infty,x\right]}$ for any $x\in\mathbb{R}$.
  The bound $0\le u(t,x)\le 1$ allows us to get the continuity of
  $g_{\varphi}$ for any $\varphi\in\L^{1}$ by approximating it with
  integrable piecewise constant functions. Finally using
  Corollary~\ref{cor1bis} one can prove that $g_{\varphi}$ is
  continuous for any $\varphi\in\L^{\infty}$ proving the $\L^{1}$
  weak continuity.
\end{proof}

  \v
  
\begin{definition}
  \label{def:class_F} {\it 
  We denote by $\mathcal{F}_{[a,b]}$ the family  of all  fluxes $f=f(t,x,u)$ that
  satisfy {\bf (F1)}, {\bf (F2)}
  for $t\in[a,b]$, and for which the following property holds.
    For any initial data $u_0 \in\mathcal{D}$,  calling $u^\ve$ the solutions  
     to~the viscous Cauchy problem \eqref{CP4},  as $\ve\to 0$ the
    integrated functions
    \begin{displaymath}
      U^{\varepsilon}(t,x)=\int_{-\infty}^{x}u^{\varepsilon}\left(t,y\right)\;dy
    \end{displaymath}   
     converge uniformly in $\left[a,b\right]\times \mathbb{R}$ to a unique limit.}
\end{definition}
\v

By Lemma~\ref{lm:compact}, if $f\in\mathcal{F}_{\left[0,T\right]}$,
then as $\ve\to 0$ the solutions $u^{\varepsilon}(t,\cdot)$ of~\eqref{CPe}
converge
weakly to a unique limit $u(t,\cdot)$ in the weak topology of
$\L^{1}\left(\mathbb{R},\mathbb{R}\right)$ for any fixed
$t\in\left[0,T\right]$.
The map $t\mapsto
u(t,\cdot)$ is continuous from $\left[0,T\right]$ into $\L^{1}\left(
  \mathbb{R},\mathbb{R}\right)$ endowed with its weak topology.

Our eventual goal is to show that $\mathcal{F}_{\left[0,T\right]}$
contains a set of flux functions of the
form $f(t,x,\omega)= F(v(t,x),\omega)$, where $F$  is smooth and
$v=v(t,x)$ is a regulated function. 
In this direction, our  main tools are the following approximation results.

\begin{lemma}
  \label{lem:timeglue}
  Given two fluxes $f_{1}\in\mathcal{F}_{\left[a,c\right]}$ and
  $f_{2}\in\mathcal{F}_{\left[c,b\right]}$ with $0\le a<c<b$, then the
  flux $f$ defined by
  \begin{displaymath}
    f\left(t,x,\omega\right)
    =
    \begin{cases}
    f_{1}\left(t,x,\omega\right)&\text{ for }t\in\left[a,c\right]\\
    f_{2}\left(t,x,\omega\right)&\text{ for }t\in\left]c,b\right]      
    \end{cases}
  \end{displaymath}
  belongs to $\mathcal{F}_{\left[a,b\right]}$.
\end{lemma}
\begin{proof}
Let an initial data
\begin{equation}  \label{uidn}
  u(0,\cdot)~=~u_0 ~\in \D
\end{equation}
be given.
For any  $\ve>0$, let $u^\ve$ 
be the corresponding solution to
\begin{equation}\label{tcpn} 
  u_t + f(t,x,u)_x~=~\ve u_{xx}\,,
\end{equation}
and define the integrated function
  $U^\ve(t,x)~=~\int_{-\infty}^x  u^\ve(t,y)dy$.
  Since in $\left[a,c\right[$ we have
  $f=f_{1}\in\mathcal{F}_{\left[a,c\right]}$,
 the limit 
$U~=~\lim_{\ve\to 0}  U^\ve$
is well defined in $\C^0(\left[a,c\right]\times\R)$ (we can change
$f$ into $f_{1}$ in $t=c$ without changing the solution at time $t=c$).

The uniform convergence implies that for any $\delta>0$, 
we can find $\ve_{o}>0$ such that 
\begin{equation}\label{comp3n}
U(c,x)-\delta~\leq~
U^\ve(c,x)~\leq~U(c,x)+\delta,\quad
\text{for all } x\in\mathbb{R} \text{ and }0<\ve<\ve_{o}.
\end{equation}
On the interval $[c, b]$,
consider the solution $\hat u^\ve$ to~\eqref{tcpn} with initial data
$\hat u^\ve(c,\cdot )~=~u(c,\cdot)$, where $u=U_{x}$ is the weak limit of
$u^{\varepsilon}$ in $\left[a,c\right]$.
The assumption $f_{2}\in \F_{[c, b]}$ implies that
the limit $\Hat U$ of the integrated functions 
$\Hat U^\ve\left(t,x\right)=\int_{-\infty}^{x}\hat u^{\varepsilon}(t,y)dy$
is well defined in $\C^0(\left[c,b\right]\times\R)$, so that, possibly
choosing a smaller $\varepsilon_{o}>0$ one has
\begin{displaymath}
  \Hat U(t,x)-\delta<\Hat U^\ve(t,x)~
  <\Hat
  U(t,x)+\delta\,\quad\text{ for all
  }\left(t,x\right)\in\left[c,b\right]
  \times\mathbb{R},\ 0<\varepsilon<\varepsilon_{o}.
\end{displaymath}

We now observe that, for $0<\ve<\ve_{o}$, $t\in\left[c,b\right]$
the functions $u^\ve$ satisfy the
same parabolic equation~\eqref{tcpn} as $\hat u^\ve$, with initial data
at $t=c$
respectively equal to $u^\ve\left(c,x\right)$ and $
u\left(c,x\right)$ whose integrated functions satisfy~\eqref{comp3n}.
By the comparison property proved in Theorem~\ref{th:comp2}, we
now obtain
for all $\varepsilon>0$ sufficiently small and
$(t,x)\in\left[c,b\right]\times\mathbb{R}$
\begin{displaymath}
  \Hat U(t,x)-2\delta<\Hat U^\ve(t,x)-\delta~\leq~
  U^\ve(t,x)~\leq~\Hat U^\ve(t,x)+\delta<\Hat
  U(t,x)+2\delta\,.
\end{displaymath}
Since $\delta>0$  was arbitrary, we thus conclude
that $U^{\varepsilon}$ converges to $\Hat U$ in
$\mathcal{C}^{0}\left(\left[c,b\right]\times \mathbb{R}\right)$.
\end{proof}

\v

\begin{theorem}
  \label{th:F_approx}
  Consider a flux $f=f(t,x,u)$ defined for $t\in[0,T]$, satisfying
  {\bf (F1)} and  {\bf (F2)}. Assume that, for any $\delta>0$,
  there exists times
  \[0<a_{1}<b_{1}< ~\cdots ~< a_{N}<b_{N}<T\]
  and flux functions $f_{i}\in\mathcal{F}_{[a_{i},b_{i}]}$
  such that  
  \begin{equation}\label{sbig} 
  T-\sum_{i=1}^{N} (b_i-a_i)~<~\delta\,,
  \end{equation} 
  and for  $i =1,\ldots,N$, 
  \begin{equation}
    \label{3.4}
    \left|f\left(t,x,\omega\right)-f_{i}\left(t,x,\omega\right)\right|<
    \delta
    \quad\text{ for any }\left(t,x,\omega\right)\in
    \left[a_{i},b_{i}\right]\times\mathbb{R}\times\left[0,1\right]\,.
  \end{equation}
  Then $f\in\mathcal{F}_{[0,T]}$.
\end{theorem}

\v
\begin{proof}
\textbf{1.}   
Fix $\delta>0$ and choose time intervals $[a_i, b_i]$ and functions $f_i$
as in the assumptions of the theorem.
Consider the new flux function
\begin{equation}\label{tilf}\tilde f\left(t,x,\omega\right)\,=\,
\begin{cases}
    f_{i}\left(t,x,\omega\right)\qquad &\hbox{if}~t\in[a_{i}, b_{i}],\ i=1,\ldots,N,\\
      0\qquad &\text{otherwise.}
\end{cases}
  \end{equation}
We claim that $\tilde f\in \F_{[0,T]}$.
Indeed, since the flux identically zero belongs
trivially to $\mathcal{F}_{\left[\bar a,\bar b\right]}$ for any
$0\le\bar a < \bar b$, it is enough to apply repeatedly
Lemma~\ref{lem:timeglue}.

\v
\n {\bf 2.} 
Fix an initial data
$u_0 \in\D$ and call $u^{\varepsilon}$ and $\tilde u^{\varepsilon}$
respectively the solutions to the Cauchy problems
\begin{displaymath}
  \begin{cases}
    u_t + f(t,x,u)_x~=~\ve u_{xx},\\
    u(0,\cdot)= u_0 ,
  \end{cases}
  \quad\mbox{and}\qquad 
  \begin{cases}
    u_t + \tilde f(t,x,u)_x~=~\ve u_{xx},\\
    u(0,\cdot)= u_0 ,
  \end{cases}
\end{displaymath}
and $U^{\varepsilon}$ and $\Tilde U^{\varepsilon}$ their integrals:
\begin{displaymath}
  U^{\varepsilon}\left(t,x\right)=\int_{-\infty}^{x}u^{\varepsilon}
  \left(t,y\right)\;dy,\qquad
  \Tilde U^{\varepsilon}\left(t,x\right)=\int_{-\infty}^{x}\tilde u^{\varepsilon}
  \left(t,y\right)\;dy.
\end{displaymath}
From point \textbf{1.} we know that $\Tilde
U^{\varepsilon}$ converges in
$\mathcal{C}^{0}\left(\left[0.T\right]\times\mathbb{R}\right)$
to a unique limit $\Tilde
U$.

By the assumption {\bf (F2)} we have $f(t,x,0)=0$,
hence by (ii) in {\bf (F1)} it follows
the uniform bound
\begin{equation}\label{fb}
\bigl|f(t,x,\omega)\bigr|~\leq~L\qquad \forall (t,x,\omega)\in
[0,T]\times\R\times [0,1].
\end{equation}
We now introduce the error function
\begin{equation}\label{ef}
\eta(t)~\doteq~
\begin{cases}  \delta\quad&\hbox{if} \quad 
t\in [a_i, b_i], ~~i=1,\ldots,N,\\
L\quad &\hbox{otherwise.}
\end{cases}
\end{equation}

By the assumption~\eqref{3.4}, the two fluxes satisfy
\begin{displaymath}
  \tilde f\left(t,x,\omega\right)-\eta\left(t\right)
  \le f\left(t,x,\omega\right) \le \tilde f\left(t,x,\omega\right)+\eta(t)
\end{displaymath}
for all $(t,x,\omega)\in [0,T]\times\R\times [0,1]$. Since
$U\left(0,x\right)=\Tilde U(0,x)$, an application of
Theorem~\ref{th:comp2} gives
\begin{displaymath}
  \bigl| U^\ve(t,x)-\Tilde U^\ve(t,x)\bigr|~\leq~\int_0^t \eta(s)\,
  ds\le \int_0^T \eta(s)\,
  ds\le \delta\left(T+L\right),\quad
  \text{ for all }(t,x)\in\left[0,T\right]\times\mathbb{R}.
\end{displaymath}
For $\varepsilon,\sigma>0$, the previous inequality implies
\begin{displaymath}
  \begin{split}
    \left\|U^{\varepsilon}-U^{\sigma}
      \right\|_{\L^{\infty}} &\le
    \left\|U^{\varepsilon}-\Tilde U^{\varepsilon}
      \right\|_{\L^{\infty}}
    +\left\|\Tilde U^{\varepsilon}-\Tilde U^{\sigma}
      \right\|_{\L^{\infty}}
    +\left\|\Tilde U^{\sigma}- U^{\sigma}
      \right\|_{\L^{\infty}}\\
    &\le \left\|\Tilde U^{\varepsilon}-\Tilde U^{\sigma}
      \right\|_{\L^{\infty}} + 2 \delta \left(T+L\right),
  \end{split}
\end{displaymath}
where the $\L^{\infty}$ norms are taken over the set $\left[0,T\right]
\times \mathbb{R}$.
Since the limit $\Tilde U^\ve\to \Tilde U$ exists
in $\mathcal{C}^{0}\left(\left[0,T\right]\times\mathbb{R}\right)$,
taking the limit as
$\sigma,\varepsilon\to 0$ in the previous inequality, we obtain
\begin{displaymath}
   \limsup_{\sigma,\varepsilon\to 0} \left\|U^{\varepsilon}-U^{\sigma}
      \right\|_{\L^{\infty}}~ \le~ 2 \delta \left(T+L\right).
\end{displaymath}
Since $\delta>0$ was arbitrary, this
implies the existence (and uniqueness) of the limit
$\lim_{\varepsilon\to 0}U^{\varepsilon}$ in
$\mathcal{C}^{0}\left(\left[0,T\right]\times\mathbb{R}\right)$,
completing the proof.
\end{proof}
\v

As we will see, Theorem~\ref{th:F_approx} implies
that $\mathcal{F}_{\left[0,T\right]}$ contains a wide class of 
discontinuous flux functions.

By the classical result of Kruzhkov,  for conservation law with smooth flux
the vanishing viscosity 
limit exist and is unique~\cite{Bbook, Dafermos, K, Se}.
An extensive body of more recent literature has dealt with fluxes of the form
 \begin{displaymath}
      f(x,\omega)=
      \begin{cases}
        f_{l}(\omega)&\text{ for }x< 0,\\
        f_{r}(\omega)&\text{ for }x>0,
    \end{cases}
  \end{displaymath}
 assuming that  the left and right fluxes $f_{l}$ and $f_{r}$ are smooth functions such that
  \begin{equation}\label{frl}
    f_{l}(0)=f_{r}(0)=0,\qquad
    f_{l}(1)=f_{r}(1).
  \end{equation}
  In this case, one can again conclude that $f\in \F_{[0,T]}$, for every $T>0$. 
A detailed proof, based on the theory of nonlinear semigroups~\cite{Cra, CL}, can be found in~\cite{GS}.
The next lemma shows that the existence and uniqueness of the weak
limit also holds
when the interface  between the two fluxes varies in time, 
under mild regularity assumptions.

\begin{lemma}
\label{th:1jump}
  Let $f_{l}(u)$ and $f_{r}(u)$ be smooth functions satisfying~\eqref{frl}.
  Let $\gamma:[0,T]\mapsto\R$ be a Lipschitz function whose derivative 
  $\dot\gamma$ coincides a.e.~with a regulated function.
  Then the flux function $f$ defined by
\begin{equation}\label{fj}
      f(t,x,\omega)~=~
      \begin{cases}
        f_{l}(\omega)&\text{ if }~x\le \gamma\left(t\right),\\
        f_{r}(\omega)&\text{ if }~x>\gamma\left(t\right),
    \end{cases}
\end{equation}
  belongs to $\mathcal{F}_{\left[0,T\right]}$.
\end{lemma}
\v
\begin{proof}
  For any initial data $u_0 \in\mathcal{D}$, let $u^{\varepsilon}$ be the solution to
  \begin{displaymath}
    \begin{cases}
      u_{t}+f(t,x,u)_{x}~=~\varepsilon u_{xx}\,,\\
      u(0,x)~=~u_0  (x),
    \end{cases}
  \end{displaymath}
  and define
  \begin{displaymath}
    \tilde u^{\varepsilon}(t,x) ~\doteq~ u^{\varepsilon}\bigl(t,x+\gamma(t)\bigr).
  \end{displaymath}
  Then $\tilde u^{\varepsilon}\in \C^0\left(\left[0,T\right],
    \L^1(\R)\right)$ is a  solution to
  \begin{displaymath}
    \begin{cases}
      u_{t}+\tilde f\left(t,x,u\right)_{x}~=~\varepsilon u_{xx}\,,\\
      u(0,x)~=~u_0  \left(x+\gamma(0)\right)\,,
    \end{cases}
  \end{displaymath}
  where the new flux $\tilde f$, which also
  satisfies assumptions \textbf{(F1)} and \textbf{(F2)},
 is  
  \begin{displaymath}
    \tilde f(t,x,\omega)~=~
    \begin{cases}
     f_{l}(\omega)-\dot \gamma(t) \,\omega  &\text{ for }x< 0,\\
  f_{r}(\omega)-\dot \gamma(t) \,\omega &\text{ for }x> 0.
    \end{cases}
  \end{displaymath}
 Using the assumption that $\dot \gamma$ is a regulated function, 
for any $\delta>0$ we can find a piecewise constant function
$\chi:[0,T]\mapsto \R$ which satisfies $\|\chi - \dot\gamma\|_{\L^\infty[0,T]}<\delta$.
If $\eta_1,\ldots,\eta_N$ are the values of $\chi$, we can find disjoint subintervals
$[a_i, b_i]\subset [0,T]$ such that 
\[
\bigl|\chi(t)-\dot\gamma(t)\bigr|~=~\bigl|\eta_i-\dot\gamma(t)\bigr|~<~\delta
\qquad\forall t\in [a_i, b_i],~~i=1,\ldots,N,
\]
and
\[
T-\sum_{i=1}^N(b_i-a_i) ~<~\delta.
\]
Consider the fluxes
\begin{displaymath}
   f_i(t,x,\omega)~=~
    \begin{cases}
     f_{l}(\omega)-\eta_i \omega  &\text{ for }x< 0\,,\\
  f_{r}(\omega)-\eta_i \omega &\text{ for }x> 0\,.
    \end{cases}
  \end{displaymath}
By the result in~\cite{GS} it follows $f_i\in \F_{[a_i,b_i]}$ for all $i=1,\ldots,N$.
This shows that the flux function $\tilde f$ satisfies
all  the assumptions of Theorem~\ref{th:F_approx}. Hence $\tilde f\in
\F_{[0,T]}$ and the integrated function
  \[\Tilde U^{\varepsilon}\left(t,x\right)=\int_{-\infty}^{x}\tilde
  u^{\varepsilon}
  \left(t,y\right)dy\]
  converges uniformly on $\left[0,T\right]\times\mathbb{R}$.
  Therefore
  \[U^{\varepsilon}\left(t,x\right)=\int_{-\infty}^{x}
  u^{\varepsilon}
  \left(t,y\right)dy
  =\Tilde U^{\varepsilon}\left(t,x-\gamma(t)\right)\] converges
  uniformly too proving that $f\in\mathcal{F}_{\left[0,T\right]}$.
\end{proof}

  \v
  The next result shows that  functions in $\F_{[0,T]}$ can be patched
  together horizontally too, provided
  that they coincide on an intermediate domain.
  
\begin{lemma}\label{l:patch}
 Consider two flux functions $f_{1}$, $f_{2}$, both satisfying
  {\bf (F1)} and {\bf (F2)}. Assume that
  \begin{itemize}
  \item
    $f_{1},\ f_{2}\in\mathcal{F}_{\left[0,T\right]}$;
  \item  There exists $\alpha<\beta$ such that
    $f_{1}(t,x,\omega)=f_{2}(t,x,\omega)$ for
    all $t\in\left[0,T\right]$, $x\in ]\alpha,\beta[\,$, and $\omega\in[0,1]$. 
  \end{itemize}
  Then the flux $f$ defined by
  \begin{equation}
    \label{eq:composite_flux}
    f\left(t,x,\omega\right)~\doteq~
    \begin{cases}
      f_{1}\left(t,x,\omega\right)&\text{ if }x< \beta\\
      f_{2}\left(t,x,\omega\right)&\text{ if }x>\alpha   
    \end{cases}
  \end{equation}
  belongs to $\mathcal{F}_{\left[0,T\right]}$.
\end{lemma}


\begin{proof} 
It is clear that the
  patched flux $f$ also satisfies the assumptions {\bf (F1)} and {\bf (F2)}. 
  It is enough to prove the Lemma with $T< (\beta-\alpha)/4L$, and
  then apply repeatedly Lemma~\ref{lem:timeglue}. 
  For any $\ve>0$,  let $u^{\varepsilon}$ be the solution
  to~\eqref{CPe} with initial data $u_0 \in\mathcal{D}$,
  and let $u_{1}^{\varepsilon}$, $u_{2}^{\varepsilon}$ be the
  solutions to
  \begin{displaymath}
    \begin{cases}
      u_{t}+f_{1}\left(t,x,u\right)_{x}=\varepsilon u_{xx}\,,\\
      u(0)=u_0 ,
    \end{cases}\qquad 
    \begin{cases}
      u_{t}+f_{2}\left(t,x,u\right)_{x}=\varepsilon u_{xx}\,,\\
      u(0)=u_0 ,
    \end{cases}
  \end{displaymath}
  respectively. As usual, we denote by $U^\ve, U_1^\ve, U_2^\ve$ the 
  corresponding integrated functions.
  By hypothesis $U_1^\ve,
  U_2^\ve$ 
  converge uniformly on $\left[0,T\right]
  \times\mathbb{R}$,
  we need to prove that $U^\ve$ too converges uniformly.

  For any $x\in\left]-\infty,\left(\alpha+\beta\right)/2\right]$ and
  $t\in\left[0,T\right]$,
  define 
  \[\xi = \beta-Lt-x>\left(\beta-\alpha\right)/4>0.\]
  Since, for $x<\beta$, $f$ coincides with $f_{1}$,
  Corollary~\ref{cor2} gives the estimate
  \begin{displaymath}
    \begin{split}
      \left|U^\ve(t,x)- U_1^\ve(t,x)\right|&=
      \Big| U^\ve(t, \beta-Lt-\xi)- U_1^\ve(t, \beta-Lt-\xi)\Big|\\
      &=\left|\int_{-\infty}^{\beta-Lt-\xi} \bigl(u^\ve(t,x) -
        u_1^\ve(t,x)\bigr)\, dx \right|\\
      &\leq 4\|u_0 \|_{\L^1}\cdot \int_{\xi/\sqrt{t\varepsilon}}^{+\infty} G(1,y)\,
      dy\leq 4\| u_0 \|_{\L^1}\cdot
      \int_{\frac{\beta-\alpha}{4\sqrt{T\varepsilon}}}^{+\infty} G(1,y)\,
      dy.
    \end{split}
\end{displaymath}
This shows that the difference between
$U^{\varepsilon}$ and $U^{\varepsilon}_{1}$
converges to zero uniformly in
$\left[0,T\right]\times\left]-\infty,\left(\alpha+\beta\right)/2\right]$.
Since by hypothesis $U^{\varepsilon}_{1}$
converges uniformly  in that region, we obtain that
$U^{\varepsilon}\left(t,\cdot\right)$
too converges uniformly there.
An entirely similar estimate yields
the uniform convergence of $U^\ve$ in 
$\left[0,T\right]\times\left[\left(\alpha+\beta\right)/2,+\infty\right[$.
\end{proof}

\v

\begin{lemma}
  \label{lemma4}
  Let $f=f(t,x,\omega)$ be a flux function satisfying
  \textbf{(F1)}, \textbf{(F2)}.   Assume that, for every bounded
  interval $[x_1,x_2]$ the function
\begin{equation}\label{cfun}
\hat f(t,x,\omega)~=~\begin{cases}
f(t,x_1,\omega)\quad &\hbox{if}~~x<x_1\,,\\
f(t,x,\omega)\quad &\hbox{if}~~x\in[x_1, x_2]\,,\\
f(t,x_2,\omega)\quad &\hbox{if}~~x>x_2\,,
\end{cases}
\end{equation}
lies in $\F_{[0,T]}$.   Then $f\in \F_{[0,T]}$ as well.
\end{lemma}

\begin{proof}
Consider any initial data $u_0 \in \D$.   Given $\delta>0$, choose 
a constant $M=M(\delta,u_0 ,L)$ as in Corollary~\ref{cor1bis}
and choose $x_1=-M-LT$ and $x_2=M+LT$ in~\eqref{cfun}.
Let $u^\ve, \hat u^\ve$ be the solutions to the Cauchy problems
 \begin{displaymath}
    \begin{cases}
      u_{t}+f(t,x,u)_{x}~=~\varepsilon u_{xx}\,,\\
      u(0)=u_0 ,
    \end{cases}\qquad 
    \begin{cases}
      u_{t}+\hat f(t,x,u)_{x}~=~\varepsilon u_{xx}\,,\\
      u(0)=u_0 ,
    \end{cases}
  \end{displaymath}
respectively.
Let $U^\ve, \Hat U^\ve$ be the corresponding integrated functions.

Since $\hat f\in \F_{[0,T]}$,
there exists the uniform limit $\lim_{\varepsilon\to 0}\Hat U^\ve
=\Hat U$. Since $f=\tilde f$ for $x\in\left[-M-LT,M+LT\right]$ the
same argument as in the proof of Lemma~\ref{l:patch} shows that
$U^{\varepsilon}$ converges to $\Hat U$ uniformly in
$\left[0,T\right]\times \left[-M,M\right]$. The choice of the constant
$M$ implies
\begin{displaymath}
  \begin{cases}
    U^{\varepsilon}\left(t,x\right)<\delta&\text{ for all
    }\left(t,x\right)\in
    \left[0,T\right]\times \left]-\infty,-M\right]\\
    U^{\varepsilon}\left(t,x\right)-U^{\varepsilon}\left(t,M\right)<\delta
    &\text{ for all
    }\left(t,x\right)\in
    \left[0,T\right]\times \left[M,+\infty\right[
  \end{cases}
\end{displaymath}
so that, for $\varepsilon,\sigma>0$
\begin{displaymath}
  \left\|U^{\varepsilon}-
    U^{\sigma}\right\|_{\C^{0}\left(\left[0,T\right]\times\mathbb{R}\right)}
 ~ \le~ 2\delta + 
  \left\|U^{\varepsilon}-
    U^{\sigma}\right\|_{\C^{0}\left(\left[0,T\right]\times\left[-M,M\right]\right)}
\end{displaymath}
and
\begin{displaymath}
  \limsup_{\sigma,\varepsilon\to 0}\left\|U^{\varepsilon}-
    U^{\sigma}\right\|_{\C^{0}\left(\left[0,T\right]\times\mathbb{R}\right)}
 ~ \le~ 2\delta.  
\end{displaymath}
Since $\delta>0$ was arbitrary, this  concludes the proof.
\end{proof}
\v
Combining the previous results, we can now prove the main theorem of this section.

\begin{theorem}\label{th:freg}  
 Let $f=f(t,x,\omega)$ be a flux function satisfying
 \textbf{(F3)}. 
      Then $f\in\F_{[0,T]}$. 
\end{theorem}

\begin{proof}
By the assumptions~\eqref{Fass}, the  flux function $f$ satisfies
  {\bf (F1)} and {\bf (F2)}.   

Fix an interval $[x_1,x_2]$.
Let $\delta>0$ be given.  Since $v$ is regulated
we can find disjoint intervals 
$[a_i, b_i]$,
Lipschitz continuous curves $\gamma_{i,j}$ and constants $\alpha_{i,j}$ 
such that all conditions \textit{(i)}--\textit{(iii)} in Definition~1 hold.

For each $i=1,\ldots,N$, let the piecewise constant function  
$\chi_i(t,x)$ be as in~\eqref{5}.
By repeatedly applying Lemma~\ref{l:patch}, we can show that the flux function
$$f_i(t,x,\omega)~\doteq~F(\chi_i(t,x), \omega)$$
lies in $\F_{[a_i, b_i]}$.   
%
In turn, an application of Theorem~\ref{th:F_approx}
shows that the function $\hat f$ in~\eqref{cfun}
lies in $\F_{[0,T]}$.
Since the interval $[x_1,x_2]$ is arbitrary, by Lemma~\ref{lemma4},
the flux function $f$ lies in $\F_{[0,T]}$ as well.
\end{proof}

\section{The strong vanishing viscosity limit}
\setcounter{equation}{0}
\label{sec:compensated}

In this section, we assume \textbf{(F3)}. Moreover
we consider the following additional hypotheses.
\begin{description}
\item[{\bf (V1)}] 
  $v(t,x)$ is a bounded measurable function whose total variation
  w.r.t.~$x$ is integrable. More precisely, for every rectangular
  domain of the form $\left[0,T\right]\times\left[x_{1},x_{2}\right]$
  one has
\begin{equation}
  \label{eq:V1}
  \int_{0}^{T}\left(\tv\left\{v\left(t,\cdot\right);\left[x_{1},x_{2}\right]\right\}\right)\;
  dt <+\infty.
\end{equation}
\item[{\bf (F4)}] 
  For each $\alpha\in\mathbb{R}$ the partial derivative
  $\omega\mapsto F_{\omega}\left(\alpha,\omega\right)$ is not constant
  on any open interval.
\end{description}
We prove that, under \textbf{(V1)}, the unique weak limit found in the
previous section is a solution to the conservation law
\begin{equation}
  \label{eq:cons_law}
  u_{t}+f\left(t,x,u\right)_{x}=0.
\end{equation}
Moreover, if we assume \textbf{(F4)} as well, the
convergence of $u^{\varepsilon}$ is in $\L^{1}\left(\left[0,T\right]\times
  \mathbb{R}\right)$. These results are obtained using a well established 
compensated compactness argument~\cite{Dafermos, Se}.

\v
For a decreasing sequence $\delta_{\nu}\to 0$,
together with the flux function $f$ in~\eqref{Fax} we also consider the mollified functions 
\begin{equation}
  \label{eq:moll_v}
f^\nu=f_{\delta_\nu},\quad f_\delta(t,x,u)=F(v_\delta(t,x),u),\quad
v_\delta(t,x) = \int_{\Omega} \rho_\delta(t-s)\rho_\delta(x-y)
v(s,y)\; dy\,ds. 
\end{equation}
Observe that, for every $\delta_\nu>0$,
the functions
$u_{*}(t,x)=0$ and $u^{*}(t,x)=1$ are solutions to $u_t +
f^{\nu}(t,x,u)_x~=\varepsilon u_{xx}$.
By the maximum principle and by
Theorem~\ref{t:21}, if we choose initial data
$
u_0 \in\mathcal{D}$  as in~\eqref{Ddef},
then the solution $u^{\varepsilon}(t,x)$ to~\eqref{CPe} satisfies
$u^{\varepsilon}(t,\cdot)\in\mathcal{D}$ for any $t\ge 0$.
Furthermore, by assumptions \textbf{(F3)} and \textbf{(V1)}, we have, 
for every $\delta, R>0$,
\begin{equation}\label{TVdelta}
\int_0^T \int_{-R}^R\;\; \sup_{\omega\in\left[0,1\right]}\bigl|  f_{\delta,x}(t,x,\omega)\bigr|\, dxdt~\leq~C_R\,,
\end{equation}
where $C_R$ is a constant depending only on $R$ and $f$ but not on $\delta$.

\v
Next, consider any smooth (not necessarily convex) entropy function $\eta=\eta(\omega)$ with
$\eta(0)=0$ and
define the corresponding
entropy flux
\begin{displaymath}
  q(t,x,\omega)~=~\int_0^{\omega} \eta'(\tilde \omega)\,
  f_{\omega}(t,x,\tilde \omega)\, d\tilde \omega\,.
\end{displaymath}
As in~\eqref{fLip}, let $L$ be a Lipschitz constant of $f$ w.r.t.~$\omega$. Then
\begin{displaymath}
  q_{\omega}(t,x,\omega)~=~\eta'(\omega)\, f_{\omega}(t,x,\omega)\,,\qquad
  \qquad \bigl|q(t,x,\omega)\bigr|~\le~ L\,\int_0^1 \bigl|\eta'(\tilde
  \omega)\bigr|\, d\tilde\omega .
\end{displaymath}

The following lemma provides the main step in the proof based on
compensated compactness.
\begin{lemma}
  \label{th:compact_estimate}
  Let the flux $f$ satisfy {\bf (F1)}, {\bf (F2)}, {\bf (F3)}
  and \textbf{(V1)}, and choose an
initial data $u_0 \in\mathcal{D}$. Then, given any decreasing
sequence $\varepsilon_{j}\to 0$, there exists
a compact set $\K\subset W^{-1,2}_{loc} (\Omega)$ such that all solutions
$u^{\varepsilon_{j}}$
to
\begin{equation}
  \label{eq:Cauchy_j}
  \left\{\bega{rl}
    u_{t}+f(t,x,u)_{x}&=~\varepsilon_{j}
    u_{xx}\,,\\[1mm]
    u(0,x)&=~u_0 (x)\,,
  \enda\right.
\end{equation}
with $0<\ve_{j}\leq 1$ satisfy
\begin{displaymath}
  \eta(u^{\varepsilon_{j}})_{t}+q(t,x,u^{\varepsilon_{j}})_{x}~\in~ \K.
\end{displaymath}
\end{lemma}

\begin{proof}
To simplify notations we drop the index $j$.
Consider the smooth solutions of the approximated equations
\begin{equation}
  \label{eq:vc1_approx}
  u^{\varepsilon,\nu}_t + f^{\nu}(t,x,u^{\varepsilon,\nu})_x~=~\varepsilon u^{\varepsilon,\nu}_{xx}\,.
\end{equation}
where $f^{\nu}$ is defined in~\eqref{eq:moll_v}.
Given an entropy  $\eta$, define the corresponding fluxes 
\begin{displaymath}
  q^{\nu}(t,x,\omega)\;\doteq\;\int_0^{\omega} \eta'(\tilde \omega)\,
  f_{\omega}^{\nu}(t,x,\tilde \omega)\, d\tilde
  \omega~=~\eta'\left(\omega\right)
  f^{\nu}\left(t,x,\omega\right)-\int_0^{\omega} \eta''(\tilde \omega)\,
  f^{\nu}(t,x,\tilde \omega)\, d\tilde
  \omega\,.
\end{displaymath}
Inequality~\eqref{TVdelta} implies a similar
estimate on the $\L^1$ norm of the partial derivative of $q^{\nu}$ w.r.t.~$x$,
namely
\begin{equation}
  \label{eq:bvqestimate}
  \int_0^T \int_{-R}^R\;\;\sup_{\omega\in\left[0,1\right]}\bigl|q^\nu_x(t,x,\omega)\bigr|\, dx dt~\leq~C'_R\,,
 \end{equation}
 where the constant $C'_R$ depends on $R$, $f$ and $\eta$ but not on $\nu$.
 
Since \eqref{eq:vc1_approx} is satisfied in a classical sense, we can
multiply both sides  by
$\eta'\bigl(u^{\varepsilon,\nu}\bigr)$ and use the chain rule to obtain
\begin{eqnarray}
&&\hspace{-2cm} 
    \eta(u^{\varepsilon,\nu})_t +
    q^{\nu}(t,x,u^{\varepsilon,\nu})_x +
    \eta'(u^{\varepsilon,\nu})f_{x}^{\nu}(t,x,u^{\varepsilon,\nu})
    -q_{x}^{\nu}(t,x,u^{\varepsilon,\nu})
 \nonumber   \\[1mm]
    &=&\varepsilon\eta(u^{\varepsilon,\nu})_{xx}-\varepsilon
    \eta''(u^{\varepsilon,\nu})\bigl(u^{\varepsilon,\nu}_{x})^{2}.
  \label{vc2}
\end{eqnarray}
Equation~\eqref{vc2} can be written as
\begin{equation}
  \label{eq:vc2_rewritten}
  \eta(u^{\varepsilon,\nu})_t +
  q^{\nu}(t,x,u^{\varepsilon,\nu})_x
~  = ~a^{\varepsilon,\nu}+b^{\varepsilon,\nu}+c^{\varepsilon,\nu}\,,
\end{equation}
with
\begin{equation}
  \label{eq:vcdefinitions}
  \begin{cases}
    a^{\varepsilon,\nu}~\doteq~
    -\eta'(u^{\varepsilon,\nu})f_{x}^{\nu}(t,x,u^{\varepsilon,\nu})+
    q_{x}^{\nu}(t,x,u^{\varepsilon,\nu}),\\
    b^{\varepsilon,\nu}~\doteq~-\varepsilon
    \eta''(u^{\varepsilon,\nu})\bigl(u^{\varepsilon,\nu}_{x}\bigr)^2,\\
    c^{\varepsilon,\nu}~\doteq~\varepsilon\eta(u^{\varepsilon,\nu})_{xx}.
  \end{cases}
\end{equation}
By Theorem~\ref{t:21} we have  $u^{\varepsilon,\nu}\to u^{\varepsilon}$ in
$Y_{T}$. In particular
\begin{displaymath}
  u^{\varepsilon,\nu}\to u^{\varepsilon}~~~ \text{ in }
\L^1(\Omega)~~~\text{ as }\nu\to+\infty,
\end{displaymath}
and since $\omega\mapsto q^{\nu}(t,x,\omega)$ is uniformly
Lipschitz, the same  argument  used in the proof of  \eqref{eq:smooth_approx2} 
now yields the convergence
$q^{\nu}(t,x,u^{\varepsilon,\nu})\to
q(t,x,u^{\varepsilon})$ in $\L^1(\Omega)$.
Hence we have the convergence
\begin{equation}\label{conv3}
  \begin{array}{rl}
\eta(u^{\varepsilon,\nu})_t +
    q^{\nu}(t,x,u^{\varepsilon,\nu})_x
    &\to ~ \eta(u^{\varepsilon})_t +
      q(t,x,u^{\varepsilon})_x\,,\\[1mm]
      \eta(u^{\varepsilon,\nu})_{xx}&\to~
\eta(u^{\varepsilon})_{xx}     \,,                                          
  \end{array}
\end{equation}
in the space of distributions.
Inserting~\eqref{conv3} in~\eqref{vc2}, one obtains
the convergence
  \begin{equation}
    \label{eq:limit_distribution}
    a^{\varepsilon,\nu}+b^{\varepsilon,\nu}~\to~ \eta(u^{\varepsilon})_t +
  q(t,x,u^{\varepsilon})_x -
  \varepsilon\eta(u^{\varepsilon})_{xx}~\doteq~ d^{\varepsilon}\,,
  \end{equation}
  again in the space of distributions.
  
Next, consider  any open set $\Omega'$ compactly contained in
$\Omega$, i.e. its closure
satisfies $\overline{\Omega'}\subset ~\, ]0,T[\;\times \;]-R,R[\;$ for
some $R>0$.
 Choose a test function $\phi(t,x)\in[0,1]$ with compact
support in $\left]0,T\right[\times \left]-R,R\right[$ and equal to $1$ on $\overline{\Omega'}$. Substitute
$\eta(s)={s^{2}}/{2}$ in~\eqref{vc2}, multiply by $\phi$,
integrate over $\Omega$,  then by parts and use~\eqref{TVdelta},
\eqref{eq:bvqestimate} to obtain
\begin{equation}
  \label{Cphi}
  \begin{split}
      &\varepsilon\int_{\Omega'}
      (u^{\varepsilon,\nu}_{x})^{2}\,dt\,dx 
      \le
    \int_{\Omega}\varepsilon
    (u^{\varepsilon,\nu}_{x})^{2}\phi\,dt\,dx\\
    &= \int_{\Omega}\bigg[\varepsilon\frac{1}{2}(u^{\varepsilon,\nu})^{2}\phi_{xx}+
    \frac{1}{2}(u^{\varepsilon,\nu})^{2}\phi_{t}+q^{\nu}(t,x,u^{\varepsilon,\nu})\phi_{x}\\
    &\qquad\qquad\qquad\qquad\qquad\qquad\qquad\qquad-
    u^{\varepsilon,\nu}f^{\nu}_{x}(t,x,u^{\varepsilon,\nu})\phi
     + \;q^{\nu}_{x}(t,x,u^{\varepsilon,\nu})\phi
    \bigg]\,dt\,dx\\
    &\le \int_{0}^{T}\int_{-R}^{R}
    \bigg[\frac{1}{2}\bigl|\phi_{xx}\bigr|+\frac{1}{2}\bigl|\phi_{t}\bigr|+\frac{L}{2}\bigl|\phi_{x}\bigr|
    +\Big(\bigl|f_{x}^{\nu}\left(t,x,u^{\varepsilon,\nu}\right)\bigr|+
    \bigl|q_{x}^{\nu}\left(t,x,u^{\varepsilon,\nu
      }\right)\bigr|\Big)\bigg]\,dt\,dx
    \\
    & \le C_\phi\,,
  \end{split}
\end{equation}
where $C_{\phi}$ is a constant which depends only on $\Omega'$ and $\phi$.
Therefore $\varepsilon(u^{\varepsilon,\nu}_{x})^{2}$
 is bounded in
$\L^1(\Omega')$ uniformly w.r.t.~$\nu$ and $\varepsilon$. Hence the
same holds for $b^{\varepsilon,\nu}$ as well. By~\eqref{TVdelta}
and~\eqref{eq:bvqestimate} it follows that $a^{\varepsilon,\nu}$ too is bounded
in $\L^1(\Omega')$, uniformly w.r.t.~$\nu$ and $\varepsilon$. Therefore
$a^{\varepsilon,\nu}+b^{\varepsilon,\nu}$ is uniformly bounded in
$\L^1(\Omega')$. This means that the distribution $d^{\varepsilon}$
in~\eqref{eq:limit_distribution} is a measure in $\Omega'$ uniformly
bounded w.r.t.~$\varepsilon$, i.e.~there exists a bounded set
$\A\subset\mathcal{M}(\Omega')$ in the space of bounded
measures in $\Omega'$ such that $d^{\varepsilon}\in \A$ for all $\varepsilon>0$.

For any $w\in W^{1,2}_0\left(\Omega'\right)$ with
$\left\|w\right\|_{W^{1,2}_0\left(\Omega'\right)}\le 1$
compute
\begin{eqnarray*}
    \int_{\Omega'}\varepsilon \eta\bigl(u^{\varepsilon,\nu}\bigr)_{xx}\, w\, dt\,dx&=&
    -\int_{\Omega'} \varepsilon \eta\bigl(u^{\varepsilon,\nu}\bigr)_x w_x\, dt\, dx\\
    &\leq&\varepsilon \|\eta_u\|_{\L^\infty}
    \left(\int_{\Omega'}\bigl(u^{\varepsilon,\nu}_{x}\bigr)^{2}\,dt\,dx\right)^{1/2} 
    \left(\int_{\Omega'} w_x^2\, dt\,dx\right)^{1/2} \\
    &\leq&\sqrt{\varepsilon} \|\eta_u\|_{\L^\infty}
    \left(\int_{\Omega'}\varepsilon\bigl(u^{\varepsilon,\nu}_{x}\bigr)^{2}\,
      dt\,dx\right)^{1/2} 
    \\
    &\leq&\sqrt{\varepsilon} \|\eta_u\|_{\L^\infty}
   \cdot (C_\phi)^{1/2},  
\end{eqnarray*}
by~\eqref{Cphi}.
This shows that $\varepsilon
\eta\bigl(u^{\varepsilon,\nu}\bigr)_{xx}\in
\sqrt{\varepsilon}B$, where $B$ is the closed ball in
$W^{-1,2}(\Omega')$ with radius $\|\eta_u\|_{\L^\infty}
    (C_\phi)^{1/2}$
    independent of $\varepsilon$ and $\nu$.
    Therefore we also have $\varepsilon\eta\bigl(u^{\varepsilon}\bigr)_{xx}\in
\sqrt{\varepsilon}B$. This implies that, as $\ve\to 0$, we have the convergence
$\varepsilon\eta(u^{\varepsilon})_{xx}
\to 0$ in $W^{-1,2}(\Omega')$.  In turn, this
implies  $\varepsilon\eta\bigl(u^{\varepsilon}\bigr)_{xx}\in \K_{1}$,  where
$\K_{1}$ is a fixed compact set in $W^{-1,2}(\Omega')$.
Finally from~\eqref{eq:limit_distribution} it follows
\begin{equation}
  \label{eq:final_inclusion}
      \eta(u^{\varepsilon})_t +
  q(t,x,u^{\varepsilon})_x ~=~ d^{\varepsilon}+
  \varepsilon\eta(u^{\varepsilon})_{xx}~\in ~\A + \K_{1}.
\end{equation}
Since the solutions $u^{\varepsilon}$ are uniformly bounded, the left hand side
of~\eqref{eq:final_inclusion} is uniformly bounded in $W^{-1,\infty}
(\Omega')$. The compactness
result stated in Lemma 16.2.2 of~\cite{Dafermos} implies
\begin{displaymath}
\eta(u^{\varepsilon})_{t} +
q(t,x,u^{\varepsilon})_{x}~\in~\hbox{compact set in }~W^{-1,2}
(\Omega')\,.
\end{displaymath}
\end{proof}
\v

We finally have the convergence theorem.

\begin{theorem}
    \label{th:strong_convergence}
Let the flux $f$ satisfy {\bf (F1)}, {\bf (F2)}, {\bf (F3)} and \textbf{(V1)},  and choose an
initial data $u_0 \in\mathcal{D}$. Let $u^{\varepsilon}$ be the
solution to the Cauchy problem~\eqref{eq:Cauchy_j}. Then the unique weak
viscosity limit
$u(t,\cdot)=\lim_{\varepsilon\to 0}u^{\varepsilon}\left(t,\cdot\right)$ is a solution
to the conservation law~\eqref{eq:cons_law}.

Moreover if the flux satisfies \textbf{(F4)} as well, then the convergence
$u^{\varepsilon}\to u$ is in $\L^{1}\left(\Omega\right)$ endowed with
its strong topology.
\end{theorem}

\begin{proof}
{\bf 1.}
  For any $(t,x)\in\Omega$ and $v,w\in[0,1]$ define
  \begin{equation}
    \label{eq:jensen_def}
    I(t,x,v,w)~\doteq~(v-w)\int_{w}^{v}\bigl[f_{\omega}
    (t,x,\omega)\bigr]^{2}\,d\omega
    -\bigl[f(t,x,v)-f(t,x,w)\bigr]^{2}.
  \end{equation}
  The following properties hold.
  \begi
  \item[(i)]
    $(v,w)\mapsto I(t,x,v,w)$ is continuous with
    $I(t,x,v,v)=0$ for any $v\in\bigl[0,1\bigr]$.
  \item[(ii)]
    $I(t,x,v,w)\ge 0$ for any $v,w\in\bigl[0,1\bigr]$.
  \item[(iii)]
    If \textbf{(F4)} holds, $I(t,x,v,w)>0$ for any $v,w\in\bigl[0,1\bigr]$ with $v\not=w$.
  \endi
  Indeed, (i) is trivial, while (ii) and (iii) follow from Jensen's
  inequality and hypothesis \textbf{(F4)}.
  Indeed, for the proof of (iii) suppose $w<v$
  (for the proof of (ii) substitute in the following inequality $>$ with $\ge$). 
  Since \textbf{(F4)} implies that $f_{\omega}(t,x,\omega)$ is not
  constant  over the interval $\omega\in[w,v]$, we have  
  \begin{align*}
    I(t,x,v,w)~&  =~(v-w)\int_w^v\bigl[f_{\omega}(t,x,\omega)\bigr]^{2}\,d\omega
    -(v-w)^{2}\left[\frac{1}{v-w}\int_w^v
    f_{\omega}(t,x,\omega)\,d\omega\right]^{2}\\
               &>~(v-w)\int_w^v\bigl[f_{\omega}
                 (t,x,\omega)\bigr]^{2}\,d\omega
                 -(v-w)^{2}\frac{1}{v-w}\int_w^v
                 \bigl[f_{\omega}(t,x,\omega)\bigr]^{2}\,d\omega\\
                           &=~0.
  \end{align*}
  \v
 \n {\bf 2.}
  In order to apply Lemma~\ref{th:compact_estimate}, fix
  $(\tau,y)\in\Omega$ and consider the following entropies and
  corresponding fluxes
  \begin{align*}
    \eta(\omega)&=\omega, &q(t,x,\omega)&=~f(t,x,\omega),\\
    \eta_{\left(\tau,y\right)}(\omega)&=f(\tau,y,\omega),&
    q_{\left(\tau,y\right)}(t,x, \omega)
                                            &=\int_{0}^{
                                              \omega}f_{\omega}(\tau,y,\tilde
                                              \omega)
                                              f_{\omega}(t,x,\tilde \omega)\,d\tilde\omega.
  \end{align*}
We claim that there exists a constant $C_{2}\ge 0$ such that
  \begin{multline}
    \label{eq:q2_and_I}
    (v-w)\bigl[q_{\left(\tau,y\right)}(t,x,v)-q_{\left(\tau,y\right)}(t,x,w)\bigr]\\\ge~
    I(t,x,v,w)+\bigl[f(t,x,v)-f(t,x,w)\bigr]^{2}
                      -C_{2}\sup_{\omega\in[0,1]}
                      \bigl|f(\tau,y,\omega)-f(t,x,\omega)\bigr|.
  \end{multline}
  Indeed, assume $w<v$. Using {\bf (F3)} we compute                  
 \[\bega{l}
    (v-w)\bigl(q_{\left(\tau,y\right)}(t,x,v)-q_{\left(\tau,y\right)}(t,x,w)\bigr)\ds
                     ~ =~(v-w)\int_{w}^{v}f_{\omega}(\tau,y,\omega)f_{\omega}(t,x,\omega)\,d\omega\\[1mm]
                  \ds  =~(v-w)\int_{w}^{v}\bigl[f_{\omega}(t,x,\omega)\bigr]^{2}\,d\omega
                  +(v-w)\int_{w}^{v}\bigl[f_{\omega}(\tau,y,\omega)-f_{\omega}(t,x,\omega)\bigr]
                  f_{\omega}(t,x,\omega)\,d\omega\\[1mm]
                    =~I(t,x,v,w)+\bigl[f(t,x,v)-f(t,x,w)\bigr]^{2}\\[1mm]
                      \quad+(v-w)\bigg[\bigl(f(\tau,y,v)-f(t,x,v)\bigr)f_{\omega}(t,x,v)
                      -\bigl[f(\tau,y,w)-f(t,x,w)\bigr]f_{\omega}(t,x,w)\\[1mm]
                 \ds   \quad -\int_{w}^{v}\bigl(f(\tau,y,\omega)-f(t,x,\omega)\bigr)f_{\omega\omega}(t,x,\omega)\,d\omega\bigg]\\  [1mm]
                    \ge~\ds
                      I(t,x,v,w)+\bigl[f(t,x,v)-f(t,x,w)\bigr]^{2}
                      -(2L+L_{2})\sup_{\omega\in[0,1]}
                      \bigl|f(\tau,y,\omega)-f(t,x,\omega)\bigr|.
  \enda
  \]
  Here the constants $L$ and $L_2$ provide upper bounds for 
  $|f_u|$ and $|f_{uu}|$, respectively.
  \v
 \n {\bf 3.}
 Let $(u^{\varepsilon_{j}})_{j\geq 1}$ be a sequence of solutions
 to~\eqref{eq:Cauchy_j} with $\varepsilon_{j}\to 0$.
 By possibly taking a subsequence and dropping the index $j$ to
  simplify the notations, we can achieve the following weak$^{*}$ convergences in
  $\L^{\infty}(\Omega)$:
  \begin{equation}
    \label{eq:weak_limits}
    \begin{cases}
    \displaystyle    u^{\varepsilon}(t,x) ~\overset{*}{\rightharpoonup}~ \bar u(t,x),\\[1mm]
 \displaystyle   f\bigl(t,x,u^{\varepsilon}(t,x)\bigr)~\overset{*}\rightharpoonup~
                                             \bar f(t,x),\\[1mm]
\displaystyle    I\bigl(t,x,u^{\varepsilon}(t,x),\bar
    u(t,x)\bigr)~\overset{*}\rightharpoonup~
                              \bar I(t,x).
\end{cases}
  \end{equation}
  Taking further subsequences (which this time may depend on
  $\left(\tau,y\right)$) we can achieve these further weak$^{*}$ convergences in
  $\L^{\infty}(\Omega)$ 
  \begin{align}
    \label{eq:weak_limits_bis}
    f\bigl(\tau,y,u^{\varepsilon}(t,x)\bigr)&~\overset{*}\rightharpoonup~ \bar f_{\left(\tau,y\right)}(t,x),
    &q_{\left(\tau,y\right)}
      \bigl(t,x,u^{\varepsilon}(t,x)\bigr)&~\overset{*}\rightharpoonup~
                                                          \bar
                 q_{\left(\tau,y\right)}(t,x).
  \end{align}
  Notice that the weak limits $\bar u$, $\bar f$, $\bar I$ in~\eqref{eq:weak_limits}
  do not
  depend on the point $\left(\tau,y\right)$. Moreover the weak limit
  $\bar u$ is unique (independent of the sequence $\varepsilon_{j}$)
  because of Theorem~\ref{th:freg} and it satisfies
  the conservation law
  \begin{equation}
    \label{eq:weak_CL}
    \bar u_{t}+\bar f\left(t,x\right)_{x}=0.
  \end{equation}
  Theorem~\ref{th:compact_estimate} now implies
  \begin{displaymath}
    u^{\varepsilon}(t,x)_{t}+f\bigl(t,x,u^{\varepsilon}(t,x)\bigr)_{x}~\in~\K\,
    ,\qquad
    f\bigl(\tau,y,u^{\varepsilon}(t,x)\bigr)_{t}+q_{\left(\tau,y\right)}\bigl(t,x,u^{\varepsilon}(t,x)\bigr)_{x}~\in~ \K,
  \end{displaymath}
  where $\K$ is a compact set in
  $W_{loc}^{-1,2}(\Omega,\mathbb{R})$. By an application of the
  \emph{div--curl lemma}, see for example  Theorem~16.2.1 in~\cite{Dafermos}, 
 one obtains
  \begin{equation}\label{eq:divcurl_result}\bega{l}
    u^{\varepsilon}(t,x)q_{\left(\tau,y\right)}\bigl(t,x,u^{\varepsilon}(t,x)\bigr)
    -f\bigl(t,x,u^{\varepsilon}(t,x)\bigr)f\bigl(\tau,y,u^{\varepsilon}(t,x)\bigr)\\[3mm]
  \qquad\qquad   \overset{*}{\rightharpoonup}~
    \bar u(t,x)\bar q_{\left(\tau,y\right)}(t,x)-
    \bar f(t,x)\bar f_{\left(\tau,y\right)}(t,x).
  \enda\end{equation}
    Setting $v=u^{\varepsilon}(t,x)$ and $w=\bar u(t,x)$
    in \eqref{eq:q2_and_I} we obtain
\begin{eqnarray*}
  &&\hspace{-2cm}   I\bigl(t,x,u^{\varepsilon}(t,x),\bar u(t,x)\bigr)+
     \bigl[f\bigl(t,x,u^{\varepsilon}(t,x)\bigr)-f\bigl(t,x,\bar u(t,x)\bigr)\bigr]^{2}\\
  &&
    -\bigl(u^{\varepsilon}(t,x)-\bar
     u(t,x)\bigr)\bigl[q_{\left(\tau,y\right)}\bigl(t,x,u^{\varepsilon}(t,x)\bigr)-
     q_{\left(\tau,y\right)}\bigl(t,x,\bar u(t,x)\bigr)\bigr]\\
     & \le& C_{2}\sup_{\omega\in[0,1]}
                      \bigl|f(\tau,y,\omega)-f(t,x,\omega)\bigr|.
\end{eqnarray*}
 This can be written as
$$\bega{l}
     I\bigl(t,x,u^{\varepsilon}(t,x),\bar u(t,x)\bigr)\\
  \qquad-\Big[u^{\varepsilon}(t,x)q_{\left(\tau,y\right)}\bigl(t,x,u^{\varepsilon}(t,x)\bigr)-
      f\bigl(t,x,u^{\varepsilon}(t,x)\bigr)f\bigl(\tau,y,u^{\varepsilon}(t,x)\bigr)\Big]\\[3mm]
    \qquad+\bigl(u^{\varepsilon}(t,x)-\bar
      u(t,x)\bigr)q_{\left(\tau,y\right)}\bigl(t,x,\bar u(t,x)\bigr)
    +\bar u(t,x)q_{\left(\tau,y\right)}\bigl(t,x,u^{\varepsilon}(t,x)\bigr)\\[3mm]
      \qquad-2f\bigl(t,x,u^{\varepsilon}(t,x)\bigr)f\bigl(t,x,\bar
         u(t,x)\bigr) +f^2\bigl(t,x,\bar u(t,x)\bigr)\\[3mm]    \ds
         \le ~C_{2}\cdot\sup_{\omega\in[0,1]}
      \bigl|f(\tau,y,\omega)-f(t,x,\omega)\bigr| \\[3mm]\qquad +
      \bigl[f\bigl(\tau,y,u^{\varepsilon}(t,x)\bigr)-f\bigl(t,x,u^{\varepsilon}(t,x)\bigr)\bigr]f\bigl(t,x,u^{\varepsilon}(t,x)\bigr)
      \\[3mm]    \le \ds~C_{3}\cdot\sup_{\omega\in[0,1]}
      \bigl|f(\tau,y,\omega)-f(t,x,\omega)\bigr|.
\enda $$
  We take the weak$^{*}$ limit in this last equation
  using~\eqref{eq:weak_limits}, \eqref{eq:weak_limits_bis} and \eqref{eq:divcurl_result} to obtain
\begin{eqnarray*}
  &&   \bar I(t,x)-\left[\bar u(t,x)\bar q_{\left(\tau,y\right)}(t,x)-
     \bar f(t,x)\bar f_{\left(\tau,y\right)}(t,x)\right]+\bar u(t,x)
     \bar q_{\left(\tau,y\right)}(t,x)\\  
  && \quad     -2 \bar f(t,x)f\bigl(t,x,\bar u(t,x)\bigr)+f\bigl(t,x,\bar
      u(t,x)\bigr)^{2}~\le ~C_{3}\sup_{\omega\in[0,1]}
      \bigl|f(\tau,y,\omega)-f(t,x,\omega)\bigr|.
\end{eqnarray*}
 Therefore
\begin{eqnarray*}
 &&  \hspace{-1cm} \bar I(t,x)+\bigl[\bar f(t,x)-f\bigl(t,x,\bar u(t,x)\bigr)\bigr]^{2}\\
   & \le& C_{3}\sup_{\omega\in[0,1]}
    \bigl|f(\tau,y,\omega)-f(t,x,\omega)\bigr|
    +\bigl|\bar f(t,x)\bigr|\bigl|\bar f_{\left(\tau,y\right)}(t,x)-\bar f(t,x)\bigr|.
\end{eqnarray*}
  Taking the weak$^{*}$ limit in
\begin{eqnarray*}
    - \sup_{\omega\in[0,1]}
    \bigl|f(\tau,y,\omega)-f(t,x,\omega)\bigr|
   & \le& f\bigl(\tau,y,u^{\varepsilon}(t,x)\bigr)-f\bigl(t,x,u^{\varepsilon}(t,x)\bigr)\\
   & \le& \sup_{\omega\in[0,1]}
    \bigl|f(\tau,y,\omega)-f(t,x,\omega)\bigr|,
\end{eqnarray*}
  we obtain
  \begin{eqnarray*}
    - \sup_{\omega\in[0,1]}
    \bigl|f(\tau,y,\omega)-f(t,x,\omega)\bigr|
  &  \le &\bar f_{\left(\tau,y\right)}(t,x)-\bar f(t,x)\\
   & \le&\sup_{\omega\in[0,1]}
    \bigl|f(\tau,y,\omega)-f(t,x,\omega)\bigr|.
  \end{eqnarray*}
  Hence for any fixed $(\tau,y)\in\Omega$, we have for
  a.e.~$(t,x)\in\Omega$
  \begin{equation}
    \label{eq:last_ineq}
    \bar I(t,x)+\bigl[\bar f(t,x)-f\bigl(t,x,\bar u(t,x)\bigr)\bigr]^{2}
   ~ \le~ C_{4}\sup_{\omega\in[0,1]}
    \bigl|f(\tau,y,\omega)-f(t,x,\omega)\bigr|.
  \end{equation}
  \v
 \n {\bf 4.} 
Call $E_{1}$  the set of Lebesgue points of the left hand side
  of~\eqref{eq:last_ineq}.  Moreover, for each $\omega\in [0,1]$  let $E_{\omega}$ be the set of Lebesgue points
  of the map $(t,x) \mapsto f(t,x,\omega)$. Defining
  \[
  E\doteq E_{1}\cap\left(\displaystyle{\bigcap_{q\in\mathbb{Q}\cap[0,1]}E_{q}}\right),
  \]
we observe that the complement  $\Omega\setminus E$ has zero measure. Take any
  $(\tau,y)\in E$ and fix $\epsilon>0$. Let
  $\F_{\epsilon}\subset \mathbb{Q}\cap [0,1]$ be a finite set such that
  $\displaystyle{\inf_{q\in
      \F_{\epsilon}}}\bigl|q-\omega\bigr|<\epsilon$
  for every $\omega\in [0,1]$. Then
  \begin{eqnarray}
      \sup_{\omega\in[0,1]}
      \bigl|f(\tau,y,\omega)-f(t,x,\omega)\bigr|
      &\le &\max_{q\in \F_{\epsilon}}
      \bigl|f(\tau,y,q)-f(t,x,q)\bigr| +
      2L\epsilon \nonumber\\
      &\le& \sum_{q\in \F_{\epsilon}}
      \bigl|f(\tau,y,q)-f(t,x,q)\bigr| +
      2L\epsilon.
        \label{eq:finite_ineq}
  \end{eqnarray}
  Let $B_{\delta}(\tau,y)$ be the disc in $\Omega$ centered
  in $(\tau,y)$ with radius $\delta>0$, hence with area $\pi \delta^2$.
 Integrating~\eqref{eq:last_ineq} and using~\eqref{eq:finite_ineq} we obtain
  \begin{eqnarray*}
    && \hspace{-1cm} 
    \frac{1}{\pi \delta^2}\int_{B_{\delta}(\tau,y)}
       \Big(\bar I(t,x)+\bigl[\bar f(t,x)-f\bigl(t,x,\bar
       u(t,x)\bigr)\bigr]^{2}\Big)\,dt\, dx\\
    &\le & {C_{4}\over\pi\delta^2} \sum_{q\in \F_{\epsilon}}
      \int_{B_{\delta}(\tau,y)}
      \bigl|f(\tau,y,q)-f(t,x,q)\bigr|\,dt\, dx+ 2C_{4}L\epsilon.
  \end{eqnarray*}
  Since $(\tau,y)$ is a Lebesgue point for the map
  $(t,x)\mapsto f(t,x,q)$, for all $q\in \F_{\epsilon}$, letting
  $\delta\to 0$ we obtain
  \begin{displaymath}
    \bar I(\tau,y)+\bigl[\bar f(\tau,y)-f\bigl(\tau,y,\bar u(\tau,y)\bigr)\bigr]^{2}
   ~ \le ~C_{4}L\epsilon\,.
  \end{displaymath}
 By the arbitrariness of $\epsilon>0$, this implies
  \begin{displaymath}
    \bar I(\tau,y)+\bigl[\bar f(\tau,y)-f\bigl(\tau,y,\bar
        u(\tau,y)\bigr)\bigr]^{2}~\le~ 0\qquad \text{ for every~
    }(\tau,y)\in E\,.
  \end{displaymath}
 Hence $\bar I(t,x)\le 0$ a.e.~in $\Omega$.
  Since $I\bigl(t,x,u^{\varepsilon}(t,x),\bar u(t,x)\bigr)\ge 0$,
  its weak$^{*}$ limit $\bar I(t,x)$ cannot  be negative.   Therefore
  \begin{displaymath}
    \bar I(t,x)=0,\qquad\text{ and }\qquad \bar
    f\left(t,x\right)=f\left(t,x,\bar
      u\left(t,x\right)\right),\quad\text{ a.e. in }\Omega.
  \end{displaymath}
  Using~\eqref{eq:weak_CL}, this implies that the unique weak
  vanishing viscosity limit $\bar u$ is a solution to the conservation
  law~\eqref{eq:cons_law}.

  Assume now that \textbf{(F4)} holds.
  Since  $I(t,x,u^{\varepsilon}(t,x),\bar
    u(t,x)\bigr)\geq 0$ for all $\ve>0$, and it converges weakly$^{*}$ to
    zero, we conclude that  it  converges  
 strongly in $\L^1_{loc}(\Omega)$. We can thus
  take a subsequence such that $I(t,x,u^{\varepsilon}(t,x),\bar
    u(t,x)\bigr)\to 0$ a.e.~in $\Omega$. Property (iii) proved at the
    beginning of the proof implies $u^{\varepsilon}(t,x)\to\bar u(t,x)$
  a.e.~in $\Omega$, completing the proof thanks to the dominated
  convergence theorem, the uniqueness of the limit $\bar u$ and
  Corollary~\ref{cor1bis} to extend the convergence to all $\L^{1}\left(\Omega\right)$.
\end{proof}

\section{Regularity of solutions to scalar conservation laws}
\label{sec:reg}
\setcounter{equation}{0}

Consider the Cauchy problem for a scalar conservation law
\begin{equation}\label{ag}
\begin{cases}
v_t + g(v)_x=0,\qquad &t\in [0,T], ~x\in \R,
\\[1mm]
v(0,x)=v_0 (x), & x\in\R.
\end{cases}
\end{equation}
To ensure that the solution $v=v(t,x)$
is a regulated function, in the sense of Definition~\ref{def:1},
we introduce the following
conditions.
\begi
\item[(C1)]  $v_0 \in \L^\infty(\R)$ and $g''(s)>0$ for all $s\in \R$.
\item[(C2)] $v_0 $ has bounded variation and 
there exists
a value $\bar s\in \R$ such that
 $g''(s)< 0$ for $s<\bar s$ and $g''(s)> 0$ for $s>\bar s$.
\endi

\begin{theorem}\label{thm51}
Let the flux function $g$ be twice continuously differentiable.
Moreover, assume that either (C1) or (C2) holds.  
Then
the unique entropy weak solution $v=v(t,x)$ of~\eqref{ag}
is a regulated function.
\end{theorem}

\begin{figure}[htbp]
\centering
\includegraphics[scale=0.37]{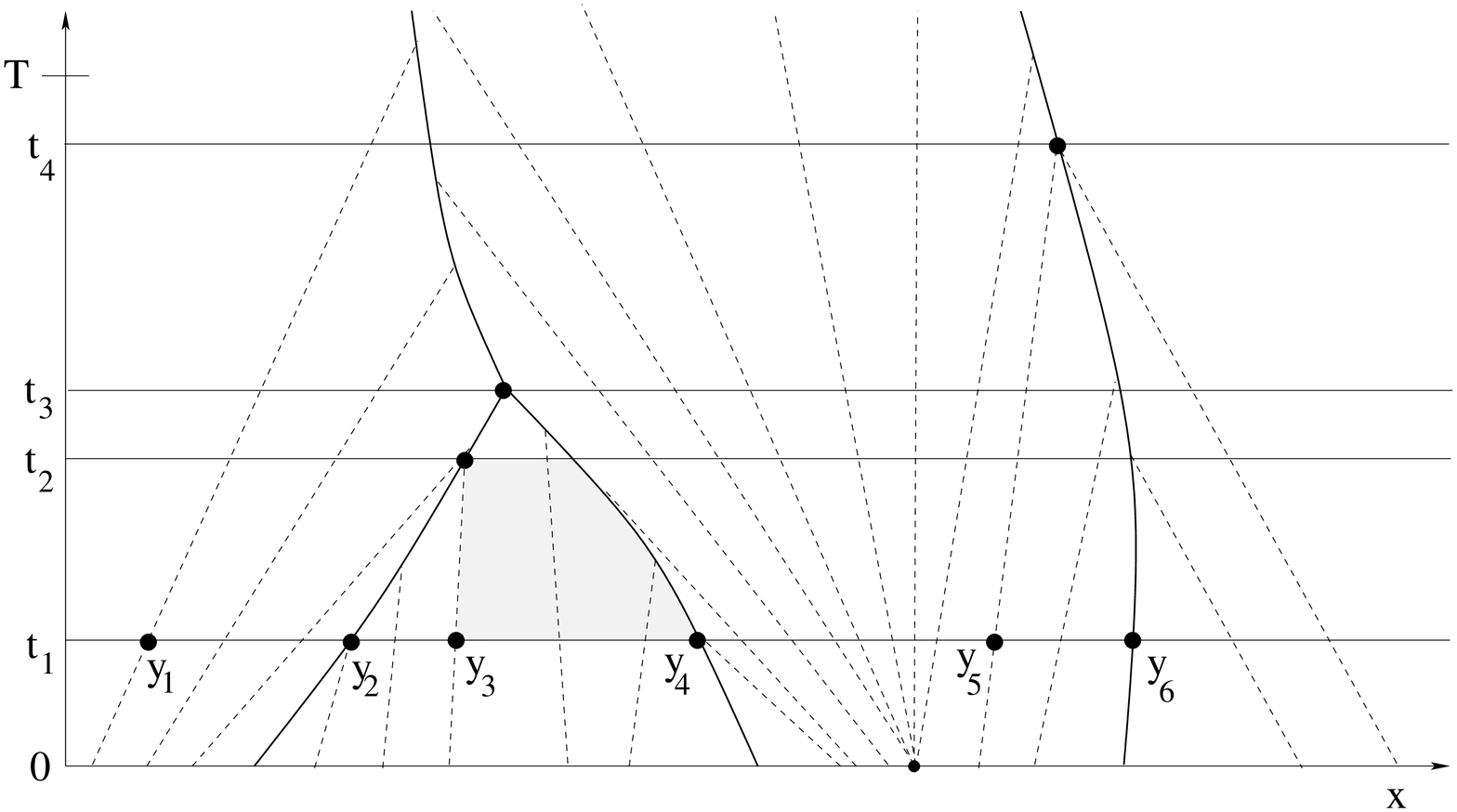}
\caption{\small Proving that, under
  condition (C1), the solution $v$ of~\eqref{ag}
  is a regulated function.
  Here we choose the  curves $\gamma_j$ to be the generalized 
  characteristics through the points $y_1,\ldots, y_6$.
  Notice that the values of the solution
  over the entire shaded region coincide with
  the values taken at time $t_1$ on the interval 
$\,]y_3, y_4[\,$.
}
\label{f:claw60}
\end{figure}

\begin{proof}
\n{\bf 1.} Assume first that the condition (C1) holds.  
To fix the ideas, assume that $v_0 (x)\in [-R,R]$
for all $x\in\R$, and 
Moreover, let $\ve>0$ and an interval $[x_1,x_2]$ be given.
By the strict convexity of the flux, at any time $t>0$ 
the solution 
$v(t,\cdot)$ satisfies Oleinik's inequality
\begin{equation}\label{OI}v(t,y)-v(t,x)~\leq~
  {y-x\over \lambda t} \,,\qquad\quad\forall x<y.\end{equation}
Choose $t_1=\ve/2$. 
Since $v(t_1,\cdot)$ has locally bounded variation, we can choose
finitely many points
\begin{equation}\label{yj} x_1-LT ~=~y_0~<~y_1<~\cdots~ <~y_N~<~y_{N+1}~=~x_2+LT\end{equation}
such that the total variation of $v(t_{1},\cdot)$ on each open interval 
$\,]y_j, y_{j+1}[\,$ is $<\ve$. 

For $j=1,\ldots,N$, 
call $t\mapsto \gamma_j(t)$ the forward generalized characteristic starting at $y_j$.
More precisely, $\gamma_j$ is the unique solution to the 
upper semicontinuous, convex valued differential inclusion
\begin{equation}\label{DI}\dot x(t)~\in~
  \Big[ g'\bigl(v(t, x(t)+))\bigr),\, g'\bigl(v(t,
  x(t)-))\bigr)\Big],
  \qquad x(t_1)~=~y_j\,.\end{equation}
We observe that, since the flux function is strictly convex, 
at any given point $(t,x)$ the right and left limits of the entropy
admissible solution $v$ 
satisfy
$$\lim_{y\to x+} v(t,y)~\leq~\lim_{y\to x-} v(t,y).$$
Oleinik's inequality~\eqref{OI} guarantees the forward uniqueness of
solutions to~\eqref{DI}.

By forward uniqueness, there can be at most $N-1$ times where
two or more of these characteristics meet. This happens when two 
shocks join together, or a genuine characteristic hits a shock. 
Let
\[
t_1~<~t_2~<~~\cdots~~<~t_m~<~t_{m+1}~=~T 
\]
be a finite set of times containing all the interaction times, for some $m\le N$.   
To satisfy the  conditions (i)--(iii) in Definition~1 we proceed as follows.
Consider the disjoint time intervals
\[[a_i, b_i]~=~\left[t_i, t_{i+1}-\ve/(2N)\right].\]
%
Define the curves 
$\gamma_{i,k}$ to be the restrictions of $\gamma_1,\ldots,\gamma_N$ to $[a_i,b_i]$.
Of course, if $\gamma_j$ and $\gamma_\ell$ coincide on $[a_i, b_i]$, they 
are regarded as one single curve.
Finally, we define the constant states as the right limits
\begin{equation}\label{aik}\alpha_{i,k}~
  \doteq~v\bigl(a_i, \gamma_{i,k}(a_i)+\bigr).
\end{equation}

It is now easy to check that all conditions (i)--(iii) in Definition~1 are satisfied.
Indeed, the set of values attained by the solution $v$ satisfies
\[
\Big\{ v(t,x)\,;~t\in [a_i, b_i],\;\gamma_{i,k-1}(t)<x<\gamma_{i,k}(t)\Big\}
\;\subseteq\;\Big\{ v(a_i,x)\,;~\gamma_{i, k-1}(a_i)<x<\gamma_{i,k}(a_i))\Big\}.
\]
Since the total variation of $v(a_i,\cdot)$ on the open interval 
$\,\bigl]\gamma_{i, k-1}(a_i),~\gamma_{i,k}(a_i))\bigr[\,$ is $<\ve$, 
this proves~\eqref{6}.

Next, we observe that the speed of a genuine characteristic 
is constant in time, while the speed of a shock is a function of bounded variation.
In all cases $\dot\gamma_j(\cdot)$ has bounded variation, hence it is a regulated function,
as required in (ii).
Finally, our construction yields
$$T-\sum_i(b_i-a_i)~= ~T-\sum_{i=1}^m (t_{i+1}-\frac{\ve}{2N} - t_i) ~
\leq~t_1 + \ve/2 ~=~\ve,$$
proving (iii).
\v

\begin{figure}[htbp]
\centering
\includegraphics[scale=0.36]{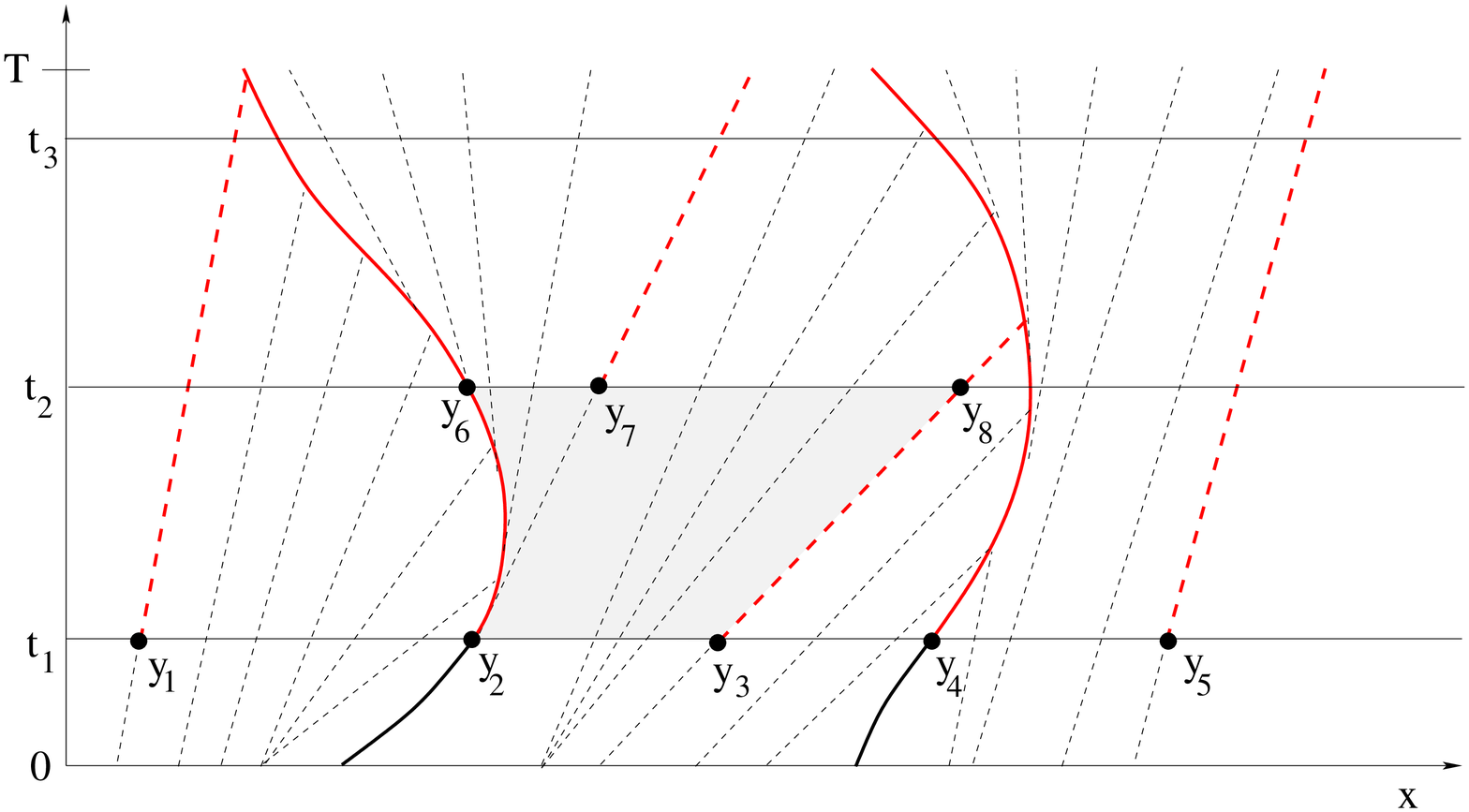}
\caption{\small If the    flux function
  has an inflection point, characteristics can originate from a
  shock, with tangential speed.  
    In this case, the values attained by the solution $v(t,x)$ over the
   shaded region (bounded by the points $y_2,y_3, y_8, y_6$) are not 
    contained in the set of values attained at time $t_1$ over the open 
    interval $]y_2, y_3[\,$.  For this reason, 
    at some time $t_2$ sufficiently close to $t_1$ we need to insert an additional 
    interface, along the characteristic starting at $y_7$.}
\label{f:claw62}
\end{figure}

{\bf 2.} Next, we consider the case where (C2) holds. 
The main difference is that 
now forward characteristics may not be unique.
Indeed, as shown in Fig~\ref{f:claw62},
characteristics can  emerge to the right of a shock, with tangential velocity.
To cope with this  issue, the previous construction can be modified as follows.

Given $\ve>0$, choose $t_1=\ve/2$. At time $t_1$, 
choose points $y_j$ as in~\eqref{yj} so that the total variation of 
$v(t_{1},\cdot)$ on each open interval 
$\,]y_j, y_{j+1}[\,$ is $<\ve/4$. 
For $j=1,\ldots,N$, 
call $t\mapsto \gamma_j(t)$ the minimal forward generalized characteristic starting at $y_j$.
More precisely, 
$$\gamma_j(t)~\doteq~\inf~\bigl\{ x(t)\,;~~x(\cdot) ~\hbox{is a solution of} ~\eqref{DI}\big\}. $$ 
Call $t_2'>t_1$ the first time where two or more of the curves $\gamma_j$ join together.
We remark that in this case it is no longer true that 
$$\Big\{ v(t,x)\,;~~t\in [t_1, t'_1],~~\gamma_{j-1}(t)<x<\gamma_{j}(t)\Big\}
~\subseteq~\Big\{ v(t_1,x)\,;~~\gamma_{j-1}(t_1)<x<\gamma_j(t_1))\Big\},$$
because of the characteristics emerging to the right of a shock. 
However, by the regularity estimates in~\cite{Glass, JS}, there exists a constant $K$ 
such that, for all $t\geq t_1$ and $x\in \R$,
\begin{eqnarray*}
v(t,x)~>~\bar s+{\ve\over 4}\qquad&\implies&\qquad v_x(t,x)~<~K,\\
v(t,x)~<~\bar s-{\ve\over 4}\qquad&\implies&\qquad v_x(t,x)~>\,-K,
\end{eqnarray*}
with $\bar s$ as in (C2).
As a consequence, we can find $\delta_0>0$ such that, on any interval of the form 
$[\tau, \tau+\delta]$ with $\tau\geq t_1$, the total strength of all rarefaction waves emerging 
tangentially from a shock is $\leq 3\ve/4$.
Choosing $t_2\doteq \min\{ t_1', t_1+\delta\}$, the total oscillation 
of $v$ over each domain
\[
\bigl\{ (t,x)\,;~~t\in [t_1, t_2],~~\gamma_{j-1}(t)<x<\gamma_j(t)\bigr\}
\]
is $\leq \ve$.
At time $t_2$ we can insert some additional points $y_k$, so that 
the total oscillation of $v(t_2,\cdot)$ on each open interval bounded by the points 
$\gamma_j(t_2)$ and $y_k$ is $\leq \ve/4$, and repeat the construction up to 
a time $t_3>t_2$, etc.

To prove  that the total number of these time intervals remains finite, we observe that 
the total strength of all rarefaction waves emerging tangentially from a shock is finite.
Indeed, these waves can be generated only when a rarefaction hits a shock form the left.
This produces a decrease in the total variation. We thus have an estimate of the form
\[
\bega{c}[\hbox{total amount of rarefaction waves emerging tangentially from a shock,}\\[1mm] 
\qquad\qquad\qquad\hbox{in the region where $|v-\bar s|>\ve/4$}]~\leq~
C\cdot\tv\{\bar v\},\enda
\]
for some constant $C$.  This ensures that the total number of additional points $y_k$ 
which we need to add during the inductive procedure is a priori bounded.
\v
Defining the constant states $\alpha_{ik}$ as in~\eqref{aik}, the remainder of the proof is 
achieved in the same way as in case (C1).
\end{proof}

\v
{\bf Remark.} As shown in Fig.~\ref{f:claw59}, 
the conclusion of Theorem~\ref{thm51} may fail if $g$ has two inflection points.
Indeed, in this case a solution $v$ can have a pair of large shocks
 splitting apart and joining together infinitely many times.
Nothing prevents the awkward situation where the two shock curves
$\gamma_1(t)\leq \gamma_2(t)$ coincide on a Cantor-like set of times, 
totally disconnected but with positive measure.
In this case, the conditions introduced in Definition~1 cannot be satisfied.
Of course, this does not  preclude the uniqueness of vanishing viscosity
solutions of the triangular  system~\eqref{TS}. It
simply yields a problem outside the scope of the present results.

\begin{figure}[htbp]
\centering
\includegraphics[scale=0.4]{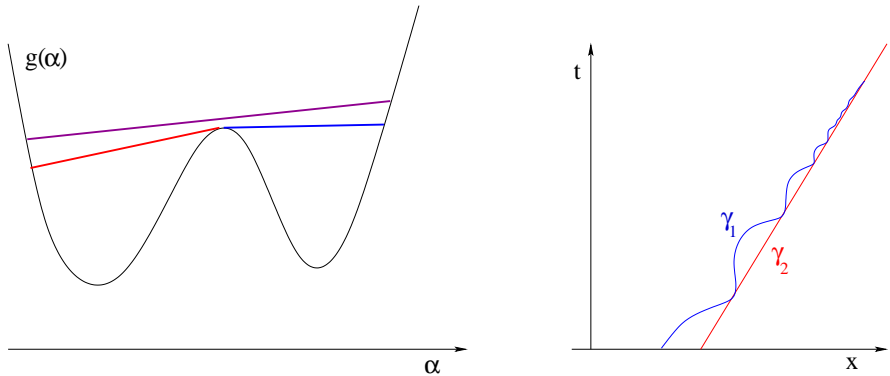}
    \caption{\small Left: a flux function $g$ with two inflection points. 
   Right:  for this flux, one can construct a solution    having two large shocks
   splitting apart and joining together infinitely many times.  }
\label{f:claw59}
\end{figure}

\section{Concluding remarks}\label{sec:concluding}
\setcounter{equation}{0}

In this paper we established the existence and uniqueness of vanishing viscosity solutions
for scalar conservation laws such as~\eqref{1}, 
where the flux function $f(t,x,\omega)=F\left(v(t,x),\omega\right)$ is
discontinuous in both $t$ and $x$. 
See~\cite{CR2} for results of well posedness for
fluxes with \textbf{BV} regularity with respect to the variable $t$.

In turn, the result yields the existence and uniqueness of solutions
for the triangular system~\eqref{TS}, under suitable assumptions on $g$.
The system~\eqref{TS} may lose hyperbolicity where the two eigenvalues
as well as the two eigenvectors coincide.
We remark that 
it is well-known that the total variation for $u$ can blow up in finite time
due to nonlinear resonances.

Our result applies beyond  the case where $v(t,x)$ 
is a solution of a scalar conservation law.
In particular, a regulated function $v(t,x)$ can have discontinuities also
along lines where $t$ is constant. 
An application is provided by polymer flooding
in two phase flow, with adsorption in rough porous media.  This leads to a system of equations
having the form
\begin{eqnarray*}
   s_t + f(s,c,\kappa)_x &=& 0,\\
   (m(c)+cs)_t + (c \, f(s,c,\kappa))_x &=& 0,\\
   \kappa_t &=& 0.
\end{eqnarray*}
The model describes an immiscible flow of water and oil phases, where polymers are dissolved
in the water phase. 
Here $s$ is the saturation of the water phase, $c$ is the fraction of polymer in the water phase,
and $\kappa=\kappa(x)$ denotes the varying porous media.
In the case of rough media, $\kappa(x)$ can be discontinuous.  
The function $f$ is the fractional flow for the water phase, where the map
$s \to f$ is typically S-shaped. 
The function $m(c)$ denotes the adsorption of polymers into the porous media,
satisfying $m' >0, m''<0$.  

A global Riemann solver for this $3\times 3$ system was constructed in~\cite{WS2017}. 
The results in the present paper suggest a possible way to solve general Cauchy problem.
The connection is best revealed using a Lagrangian coordinate system $(\phi,\psi)$, defined as
\[\phi_x=-s,\qquad \phi_t=f(s,c,\kappa), \qquad \phi(0,0)=0,\qquad \psi=x.\]
In these coordinates, the equations take the form
\begin{eqnarray*}
\left(\frac{s}{f(s,c,\kappa)}\right)_\phi -\left(\frac{1}{f(s,c,\kappa)}\right)_\psi &=& 0,\\
m(c)_\phi + c_\psi &=& 0,\\
\kappa_\phi &=& 0.
\end{eqnarray*}
We observe that the equations for $\kappa$ and $c$ are both decoupled, and can be 
solved separately. 
Treating $\phi$ as a time and $\psi$ as a space variable,
the solution $c(\phi,\psi)$ is a regulated function,
while the jumps in $\kappa$ occur along  lines parallel to the $\phi$ axis. 
The system can thus be reduced to the first equation. This is 
a scalar conservation law where the flux depends on time and space in a regulated way.
Details will be given in a future work. 

\v

{\bf Acknowledgment.} The research of the first author was partially 
supported by NSF, with grant DMS-1411786: Hyperbolic Conservation Laws
and Applications.
The research of the second author was partially 
supported by the PRIN~2015 project \emph{Hyperbolic Systems of
  Conservation Laws and Fluid Dynamics: 
  Analysis and Applications}. The second author
acknowledges the hospitality of the Department of Mathematics, Penn
State University - May 2017.
\v

\end{document}